\documentclass[reqno, 11pt]{amsart}
\usepackage{amssymb}
\usepackage{amsmath}
\usepackage{mathtools}
\usepackage{ifpdf}
\usepackage{vmargin}
\usepackage{epstopdf}

\usepackage{xcolor}
\usepackage{setspace}

\ifpdf 
\usepackage[pdftex]{hyperref}
\else
\usepackage[hypertex]{hyperref}
\fi

\numberwithin{equation}{section}

\newcommand{\eps}{\varepsilon}

\newcommand{\Z}{{\mathbb{Z}}}
\newcommand{\N}{\mathbb{N}}
\newcommand{\R}{\mathbb{R}}
\newcommand{\C}{\mathbb C}

\newcommand{\LLL}{\mathcal{L}}
\newcommand{\MMM}{\mathcal{M}}

\newcommand{\YYY}{\mathcal{Y}}
\newcommand{\ZZZ}{\mathcal{Z}}

\theoremstyle{plain}
\newtheorem{theorem}{Theorem}
\newtheorem{proposition}[theorem]{Proposition}
\newtheorem{lemma}[theorem]{Lemma}
\newtheorem{corollary}[theorem]{Corollary}

\theoremstyle{definition}

\newtheorem{remark}[theorem]{Remark}

\numberwithin{equation}{section}
\numberwithin{theorem}{section}

\numberwithin{equation}{section}

\renewcommand{\Re}{\mathrm{Re}}
\renewcommand{\Im}{\mathrm{Im}}

\newcommand{\bee}{\begin{eqnarray*}}
	\newcommand{\eee}{\end{eqnarray*}}

\numberwithin{equation}{section}
\numberwithin{theorem}{section}

\begin{document}
	\onehalfspacing
	
	\title[Dynamics of radial threshold solutions for Hartree equation]
	{Dynamics of radial threshold solutions for generalized energy-critical Hartree equation}
	
	\author[]{Xuemei Li}
	\address{\hskip-1.15em Xuemei Li
		\hfill\newline Laboratory of Mathematics and Complex Systems,
		\hfill\newline Ministry of Education,
		\hfill\newline School of Mathematical Sciences,
		\hfill\newline Beijing Normal University,
		\hfill\newline Beijing, 100875, People's Republic of China.}
	\email{xuemei\_li@mail.bnu.edu.cn}
	
	\author[]{Chenxi Liu}
	\address{\hskip-1.15em Chenxi Liu
		\hfill\newline Laboratory of Mathematics and Complex Systems,
		\hfill\newline Ministry of Education,
		\hfill\newline School of Mathematical Sciences,
		\hfill\newline Beijing Normal University,
		\hfill\newline Beijing, 100875, People's Republic of China.}
	\email{liuchenxi@mail.bnu.edu.cn}
	
	\author[]{Xingdong Tang}
	\address{\hskip-1.15em Xingdong Tang
		\hfill\newline School of Mathematics and Statistics, \hfill\newline Nanjing Univeristy of Information Science and Technology,
		\hfill\newline Nanjing, 210044,  People's Republic of China.}
	\email{txd@nuist.edu.cn}
	
	\author[]{Guixiang Xu}
	\address{\hskip-1.15em Guixiang Xu
		\hfill\newline Laboratory of Mathematics and Complex Systems,
		\hfill\newline Ministry of Education,
		\hfill\newline School of Mathematical Sciences,
		\hfill\newline Beijing Normal University,
		\hfill\newline Beijing, 100875, People's Republic of China.}
	\email{guixiang@bnu.edu.cn}

	

	\begin{abstract} In this paper, we study long time dynamics of radial threshold solutions for the focusing, generalized energy-critical Hartree equation and classify all radial threshold solutions. The main arguments are the spectral theory of the linearized operator, the modulational analysis and the concentration compactness rigidity argument developed by T. Duyckaerts and F. Merle to classify all threshold solutions for the energy critical NLS and NLW in \cite{DuyMerle:NLS:ThresholdSolution, DuyMerle:NLW:ThresholdSolution}, later by D. Li and X. Zhang in \cite{LiZh:NLS, LiZh:NLW} in higher dimensions.  The new ingredient here is to solve  the nondegeneracy of positive bubble solutions with nonlocal structure in $\dot H^1(\R^N)$ (i.e. the spectral assumption in \cite{MiaoWX:dynamic gHartree}) by the nondegeneracy result of positive bubble solution in $L^{\infty}(\R^N)$ in \cite{LLTX:Nondegeneracy} and the Moser iteration method in \cite{DiMeVald:book}, which is related to the spectral analysis of the linearized operator with nonlocal structure, and plays a key role in the construction of the special threshold solutions, and the classification of all threshold solutions.  
	\end{abstract}
	
	\maketitle
	
	
	
	\section{Introduction}\label{introduction}
	
	In this paper, we consider the focusing, generalized energy-critical Hartree equation
	\begin{equation}\label{har}
		i \partial_t u + \Delta u  + \left( I_{\lambda}*|u|^p\right)|u|^{p-2} u =0,\quad
		(t,x)\in \R \times \R^N,
	\end{equation}
	where $ N \geq 3$, $0< \lambda< N$, $\lambda\leq 4$, $p=\frac{2N-\lambda}{N-2}\geq 2$ and the Riesz-potential stands for
	$$
	I_{\lambda}(x):=\frac{\Gamma(\frac{\lambda}{2})}{\pi^{\frac{N}{2}}\Gamma(\frac{N-\lambda}{2})2^{N-\lambda}|x|^{\lambda}}.
	$$
	
	The generalized Hartree equation can be considered as a classical limit of a field equation and it describes a quantum mechanical non-relativistic many-boson system interacting through a two body potential $V(x)=\frac{1}{|x|^{\lambda}}$ , see \cite{Ginibre}. As for the mean-field limit of many-body quantum systems, we can refer to the work by Hepp in  \cite{Hepp}, see also \cite{Bardos, C.Bardos, Frohlich, Ginibre,  Spohn}. Lieb and Yau use it to develop the theory for stellar collapse in \cite{Lieb} .
	
	The equation \eqref{har} enjoys several symmetries: If $u(t,x)$ is a solution, then
	\begin{enumerate}
		\item[(a)]  so is $\delta^{\frac{N-2}{2}}u(\delta^{2} t,
		\delta x)$, $\delta> 0$ by scaling invariance;
		\item[(b)]so is $u(t+t_0, x)$ for $t_0\in \R$  by time translation invariance;
		\item[(c)]so is $u(t, x+x_0)$ for $x_0\in \R^N$ by spatial translation invariance;
		\item[(d)] so is $e^{i\theta_0}u(t,x)$, $\theta_0\in \R$  by phase rotation invariance;
		\item[(e)] so is $\overline{u(-t, x)}$  by time reversal.
	\end{enumerate}
	
	The Cauchy problem for \eqref{har} was developed by A.~K. ~Arora and S.~Roudenko  in \cite{Arora: wellposedness, Arora:LWP, Arora:local wellposedness, Caz:book}. Namely, if $u_0 \in \dot H^1(\R^N)$, there
	exists a unique solution defined on a maximal interval
	$I=\left(-T_-(u), T_+(u)\right)$ and the energy
	\begin{equation}\label{energy}
		\aligned  E(u(t)):=\frac12 \int_{\R^N} \big| \nabla u(t,x)\big|^2 dx
		-\frac{1}{2p} \int_{\R^N} \left( I_{\lambda}*|u|^p\right)\big|u(t,x)\big|^p \ dx
		\endaligned
	\end{equation}
	is conserved by time translation invariance of equation \eqref{har}. The exponent
	$
	p=\frac{2N-\lambda}{N-2}
	$
	is energy-critical exponent since the scaling transform
	\begin{equation*}
		u(t,x)\rightarrow
		\delta^{\frac{N-2}{2}}u\left(\delta^{2} t,
		\delta x\right), \; \delta>0
	\end{equation*}
	makes the equation \eqref{har} and the energy
	\eqref{energy} invariant.

	For long time behavior of solutions for \eqref{har}, the ground state plays an important role, which is the solution of the elliptic equation with nonlocal potential
	\begin{equation}\label{gs}
		\aligned -\Delta W (x)= \left( I_{\lambda}*W^p\right)\;W^{p-1}, \quad x\in \R^N.
		\endaligned
	\end{equation}
	
	According to \cite{Du:ground state}, we have the variational characterization of $W$, which can be proved by combining sharp Sobolev inequality \cite{Aubin, LiebL:book,	Talenti:best constant} with sharp Hardy-Littlewood-Sobolev inequality \cite{Lieb:sharp constant for HLS, LiebL:book}. Recently, the authors give an alternative proof of the existence of the extremizer of this sharp Hardy-Littlewood-Sobolev inequality in $\R^N$ by using the stereographic projection and sharp Hardy-Littlewood-Sobolev inequality on the sphere  in \cite{LLTX:Nondegeneracy}. That is, 
	
	\begin{proposition}[\cite{Du:ground state, LLTX:Nondegeneracy}]\label{SharpConstant} Let  $ N \geq 3$, $0< \lambda< N$, $\lambda\leq 4$, and 	$
		p=\frac{2N-\lambda}{N-2}\geq 2$. Then for any $f \in \dot H^1(\R^N), $ we have
		\begin{align*}
			\left(\int_{\R^N}
			\left( I_{\lambda}*|f|^p\right)\; |f(x)|^p \right)^{\frac{1}{p}}\leq & \; C(N,\lambda) \big\|\nabla f \big\|^2_2,
		\end{align*}
	where
$
	C(N,\lambda)=\left[(\frac{1}{2\sqrt{\pi}})^{N-\lambda}\frac{\Gamma(\frac{\lambda}{2})}{\Gamma(N-\frac{\lambda}{2})}\big(\frac{\Gamma(N)}{\Gamma(\frac{N}{2})}\big)^{\frac{N-\lambda}{N}}\right]^{\frac{1}{2p}}\cdot \left[\frac{N(N-2)}{4}2^{2/N}\pi^{\frac{N+1}{N}}\Gamma(\frac{N+1}{2})^{-\frac{2}{N}}\right]^{-\frac{1}{2}}
$
		and if
		\begin{equation*}\aligned \left(\int_{\R^N}
			\left( I_{\lambda}*|f|^p\right)\; |f(x)|^p \right)^{\frac{1}{p}}= & \; C(N,\lambda) \big\|\nabla f \big\|^2_2,  \endaligned
		\end{equation*}
		then there exist $\lambda_0>0$, $x_0 \in \R^N$, $\theta_0 \in [0,
		2\pi)$, such that
		\begin{equation*}
			\aligned f(x)= e^{i\theta_0} \lambda^{-\frac{N-2}{2}}_0W\left(
			\lambda^{-1}_0(x+x_0)\right),
			\endaligned
		\end{equation*}
	   where
	   \begin{equation}\label{w function}
	   	W=\left[N(N-2)\right]^{\frac{N-2}{4}}\big(\frac{1}{1+|x|^2}\big)^{\frac{N-2}{2}}.
	   \end{equation}

	\end{proposition}
	
	As for the solution of \eqref{har} with energy below the threshold $E(W)$, C. Guzmán and C. Xu make use of the concentration compactness rigidity argument, which is originally developed by C.~Kenig and F.~Merle in \cite{Kenigmerle:H1 critical
		NLS, Kenig-merle:wave},   (see  also  \cite{MiaoXZ:09:e-critical radial Har}) to show that
	\begin{theorem}[\cite{XuCb:scattering threshold}]\label{belowthreshold}
		Let  $ N= 3$, $0< \lambda< N$,  $p=\frac{2N-\lambda}{N-2} \geq 2$, $W$ be the ground state solution of \eqref{gs} and $u$ be a radial solution of \eqref{har} in $\dot H^1(\R^N) $ such that 
		$$  E(u_0)< E(W)\;\text{and}\;\;\|u_0\|_{\dot{H}^1(\R^N)}<\|W\|_{\dot{H}^1(\R^N)}.$$
		Then the solution $u$ is global in time and scatters both forward and backward in time, i.e. there exists $u_{\pm}\in \dot H^1(\R^N)$, such that 	$$\lim_{t\rightarrow \pm \infty}\big\| u(t)-e^{it \Delta }u_{\pm}\big\|_{\dot H^1(\R^N)} = 0.$$
	\end{theorem}
	
	\begin{remark}\label{belowthreshold:blow up}
		Using the Virial identity, we can also show the blow-up result:
		if  $E(u_0)< E(W)$, $\big\|\nabla u_0\big\|_{L^2}> \big\|\nabla W\big\|_{L^2}$ and  $x\cdot u_0 \in L^2(\R^N)$, then  the solution blows up at finite time $T_{\pm}$, which can be shown from \cite{MiaoXZ:09:e-critical radial Har}. So far, there have been many other dynamics results of the Hartree equation. In the mass and energy critical cases, please refer to \cite{KriLR:m-critcal Har, LiMZ:e-critical Har,  LiZh:har:class, MiaoXZ:09:e-critical radial Har, MiaoXZ:09:m-critical Har, MiaoXZ:10:blowup, Na99d, YRZ:L2}, and we can refer to \cite{AR:mass-energy2, AR:mass-energy, CaoG, GiV00, KriMR:mass-subcritical har, MXY:defocusing, MXZ:defocusingradial, MXZ:defocusingn, ZhouTao: Threshold} for other cases. 
	\end{remark}
	
	In this paper, we follow the argument by T.~Duyckaerts and F. Merle in \cite{DuyMerle:NLS:ThresholdSolution,DuyMerle:NLW:ThresholdSolution}, D. Li and X. Zhang in \cite{LiZh:NLS, LiZh:NLW} and consider the dynamics of radial solutions of \eqref{har} with threshold energy $E(W)$. Our goal is to classify
	the threshold solutions with 
	\begin{equation*}
		\aligned u_0 \in \dot H^1(\R^N) \, \text{is radial, and}\;\; E(u_0)=E(W).
		\endaligned
	\end{equation*}
	For the case $p=2$ and $\lambda=4<N$, C. Miao, Y. Wu and the fourth author study the dynamics of the radial threshold solution and classify the corresponding radial solutions in the energy space $\dot H^1(\R^N)$ under the spectral assumption (i. e. nondegeneracy of the ground state ) in \cite{MiaoWX:dynamic gHartree}. Based on the results by authors in  \cite{LLTX:Nondegeneracy, LTX:Nondegeneracy}, the motivation in this paper is to remove the spectral assumption in \cite{MiaoWX:dynamic gHartree} and solve the nondegeneracy of the positive bubble solutions of \eqref{har}  in the energy space $\dot H^1(\R^N)$. That is,
	
	\begin{proposition}\label{prop:nondegeneracy}
		Let $ N \geq 3$, $0< \lambda< N$, $\lambda\leq 4$ and $p=\frac{2N-\lambda}{N-2}\geq 2$. Then the nontrivial solution $U_{c,\delta,x_0}(x):=\frac{c}{\left(1+\delta^2|x-x_0|^2\right)^{\frac{N-2}{2}}}$ of the  equation
		$$-\Delta u  - \left( I_{\lambda}*|u|^p\right)|u|^{p-2} u =0,\quad x\in \R^N$$
		with constant sign is non-degenerate. More precisely, any energy solution to the linearized equation 
		\begin{equation}\label{elliptic equation}
			-\Delta \varphi =p \left[ I_{\lambda}*(|u|^{p-1}\varphi)\right]u^{p-1} +(p-1)\left( I_{\lambda}*|u|^{p}\right)u^{p-2}\varphi
		\end{equation}
		must be the linear combination of the functions $\varphi_1$, $\cdots$, $\varphi_N$  and $\varphi_{N+1}$ defined by
		$$\varphi_j(x):=\frac{\partial U_{1,1,0}}{\partial x_j}(x), 1\leq j\leq N,  \text{~and ~ }\varphi_{N+1}(x):=\frac{N-2}{2}U_{1,1,0}(x)+x\cdot \nabla U_{1,1,0}(x). $$
	\end{proposition}
	
	\begin{remark} \label{rem:nondeg} The uniqueness and nondegeneracy of the ground state itself are very interesting and important problems, we can refer to \cite{FrankLen:frac lap, FrankLenSi:frac lap, LWei:KPI, Teschl:book} and references therein, and they also play a key role  in long-time dynamics of nonlinear dispersive equations (for example, NLS, NLW, Hartree, Boson star, etc.), especially in  spectral analysis of the linearized operator,  the construction of the threshold solutions, and the classification of the threshold solutions. Now we give some remarks about Proposition \ref{prop:nondegeneracy} as following.
		\begin{enumerate}
			\item[(a)] In \cite{LTX:Nondegeneracy}, The first, third and fourth authors show the nondegeneracy of the ground state in the Newtonian potential case $\lambda=4, N=6$ by using the spherical harmonic expansion of the Newtonian potential, which is related to Gegenbauer identity in \cite{ArkHan:book, Lenz:unique}.
			\item[(b)] Motivated by R.~Frank and E.~Lieb in \cite{FrankL:rearrange, FrankL:HG}, the authors make use of the spherical harmonic decomposition and the Funk-Hecke formula of the spherical harmonic functions  (see also \cite{ArkHan:book, DaiXu:book}) to obtain the nondegeneracy of the ground state in $L^{\infty}(\R^N)$ in \cite{LLTX:Nondegeneracy} for the general case .
			\item[(c)]  Based on the nondegeneracy of positive bubble solution in $L^{\infty}(\R^N)$ in \cite{LLTX:Nondegeneracy}, we only need show $L^{\infty}(\R^N)$-regularity of $\dot H^1(\R^N)$ solution of equation \eqref{elliptic equation}   to complete the proof of Proposition \ref{prop:nondegeneracy}, which will be shown by the Moser iteration method  (please refer to \cite{DiMeVald:book}, and references therein) in Appendix \ref{appen regularity}.
		\end{enumerate}
	\end{remark}
	
	After showing the nondegeneracy result in Proposition \ref{prop:nondegeneracy}, we can follow the standard argument in \cite{DuyMerle:NLS:ThresholdSolution, DuyMerle:NLW:ThresholdSolution, LiZh:NLS, LiZh:NLW, MiaoWX:dynamic gHartree} to study the dynamics of threshold solution and classify all radial threshold solutions. The classification of threshold solutions is more complicate than  that in Theorem \ref{belowthreshold}. In fact, the ground state $W$ is a new solution which
	doesn't satisfy the conclusion in Theorem \ref{belowthreshold}.
	Besides $W$, there also exist two other special radial threshold solutions
	$W^{\pm}$.
	
	\begin{theorem}\label{threholdsolution} There exist two radial solutions $W^\pm$ of \eqref{har} on the maximal lifespan interval $(T_{-}(W^{\pm}), T_{+}(W^{\pm}))$ with
		initial data $W^{\pm}_{0} \in \dot H^1(\R^N)$  such that
		\begin{enumerate}
			\item[\rm(a)] $E(W^\pm)=E(W)$, $T_+(W^{\pm})=\infty$ and
			\begin{equation*} \aligned
				\lim_{t\rightarrow +\infty}W^{\pm}(t) =  W \; \text{in} \; \dot H^1.
				\endaligned
			\end{equation*}
			
			\item[\rm(b)] $\big\|\nabla W^-_0\big\|_{2} < \big\|\nabla W\big\|_{2}$, $T_-(W^{-})=\infty$ and $W^-$ scatters for the negative time.
			
			\item[\rm(c)] $\big\|\nabla W^+_0\big\|_{2} > \big\|\nabla W\big\|_{2}$. Moreover, if $N\geq 5$, $T_-(W^{+})<\infty$.
		\end{enumerate}
	\end{theorem}
	
	\begin{remark}
		There is no $L^2$-regularity for $W^+$ and $W$ when $N=3, 4$ from the construction in Section \ref{S:existence} and there is additional $L^2$-regularity for threshold solution in Proposition \ref{expdecay:supercase} in Section \ref{S:convergence:sup}.  We also expect that $T_-(W^{+})<\infty$ for the case $N=3, 4$.
	\end{remark}
	
	Next, we classify all radial threshold solutions as follows. 
	
	\begin{theorem}\label{classification} Suppose $u_0 \in \dot
		H^1\big(\R^N\big)$ is radial, and such that
		$ E(u_0 )=E(W )$. 	Let  $u$  be a solution of \eqref{har} with initial data $u_0 $
		and $I$ be the maximal interval of existence. Then the following
		holds:
		\begin{enumerate}
			\item[\rm(a)] If $\big\|\nabla u_0\big\|_{2} < \big\|\nabla W \big\|_{2}$, then $I=\R$. Furthermore,
			either $u=W^{-}$ up to the symmetries of the equation, or
			$\big\|u\big\|_{L^{2p}_tL^{\frac{2Np}{Np-2p-2}}_{x}}<\infty$.
			
			\item[\rm(b)] If $\big\|\nabla u_0\big\|_{2} = \big\|\nabla W \big\|_{2}$, then  $u= W$ up to the symmetries of the equation.
			
			\item[\rm(c)] If $\big\|\nabla u_0\big\|_{2} > \big\|\nabla W \big\|_{2}$, and $u_0 \in L^2(\R^N)$, then
			either $u=W^{+}$ up to the symmetries of the equation, or $I$ is
			finite.
		\end{enumerate}
	\end{theorem}

    \begin{remark}\label{case(c)N=34}
     When $N=3, \, 4$	in the supercritical case (c), the above result shows that any $L^2(\R^N)$-solution blows up in both time directions from the fact that $W^+ \not \in L^2(\R^N)$.
    \end{remark}

	\begin{remark} Because of radial assumption, the symmetries of the equation here  refer to the symmetries under scaling, time translation, phase rotation and time reverse. The above result is a generalization of the classification of radial threshold solutions by Miao, Wu and the fourth author in \cite{MiaoWX:dynamic gHartree}, which consider the  case  $\lambda=4 <N $, and $p=2$.  For $\dot H^1(\R^N)$ critical NLS,  Q. Su and Z. Zhao classify nonradial, subcritical threshold solution in  \cite{SZhao}.  In addition, for subcritical NLS, we can refer to \cite{DuyRouden:NLS:ThresholdSolution}  by T. Duyckaerts and S. Roudenko,  \cite{CFRoud:threshold} by L. Campos, L. G. Farah and S. Roudenko.  and T.~Zhou also classifies threshold solutions of subcritical Hartree equation under spectral assumption in \cite{ZhouTao: Threshold}. 
	\end{remark}

	Last, the paper is organized as follows. In Section \ref{S:preli}, we recall the Cauchy problem and the properties of the ground state. We also state the spectral properties of the linearized operator $\mathcal{L}$ around $W$, which rely on the nondegeneracy of positive bubble solutions  in Proposition \ref{prop:nondegeneracy}. In Section \ref{S:existence}, we prove Theorem \ref{threholdsolution} except for the negative time behavior, in particular, we construct two special solutions $W^{\pm}$ of \eqref{har} except for the negative time behavior. In Section \ref{S:modulation}, we discuss the modulational stability around $W$. In Section \ref{S:convergence:sup} and Section \ref{S:convergence:sub}, we study the solutions with initial data satisfying Theorem \ref{classification} part (a) and (c), and also obtain the proof of Theorem \ref{threholdsolution} for the negative time behavior. In Section \ref{S:uniqueness}, we analyze the property of the exponentially decay solution of the linearized equation to establish the uniqueness of the special solutions and this will imply the proof of Theorem \ref{classification}. In appendix \ref{appen regularity}, we show $L^{\infty}(\R^N)$-regularity of $\dot{H}^{1}(\R^N)$ solution of equation \eqref{elliptic equation} by the Moser iteration method. In appendix \ref{appen CoerH} and appendix \ref{appen-spectralprop}, we prove the positivity and spectral properties of the linearized operator in Proposition \ref{coerH} and Proposition \ref{spectral}.\qed
	
	\noindent \subsection*{Acknowledgements.}
	G. Xu was partly supported by National Key Research and Development Program of China (No.~2020YFA0712900) and by NSFC (No. ~12371240). X. Tang were partly supported by NSFC (No.~12001284). 
	
	%
	%
	%
	%
	
	\section{Preliminaries}\label{S:preli}
	
	In this section, we give some preliminary results for the context.
	
	\subsection{The linear estimates and the Cauchy problem.} In this subsection, we firstly recall some
	results on the Strichartz estimate and Cauchy problem of \eqref{har}. Let $I$ be an
	interval, and denote
	\begin{equation*}
		\aligned Z(I):=L^{2p}\left(I; L^{\frac{2Np}{Np-2p-2}}({\R^N})\right), \;
		S(I):= L^{2p}\left(I; L^{\frac{2Np}{Np-2}}({\R^N})\right),
		\endaligned
	\end{equation*}
	\begin{equation*}
		\aligned 
		N(I):=L^{\frac{2p}{2p-1}}\left(I; L^{\frac{2Np}{Np+2}}({\R^N})\right),\;l(I):= Z(I)\cap L^{2p}\left(I; \dot{W}^{1, \frac{2Np}{Np-2}}({\R^N})\right),
		\endaligned
	\end{equation*}
	\begin{equation*}
		\aligned 
		\big\|u\big\|_{l(I)}:= \big\|u\big\|_{Z(I)}+\big\|\nabla
		u\big\|_{S(I)}.	
		\endaligned
	\end{equation*}
	Here we write $N(I)$ to represent the inhomogeneous space-time norm on $I\times \R^N$ and the single $N$ to represent spatial dimension in the context without confusion.
	
	A solution of \eqref{har} on an interval $I$ with $0\in I$ is a
	function $u\in C^0(I, \dot H^1(\R^N))$ such that $u\in Z(J)$ for all
	interval $J\Subset I$ and
	\begin{equation*}
		\aligned u(t)= e^{it\Delta } u_0 + i \int^t_0 e^{i(t-s)\Delta
		}\left( I_{\lambda}*|u(s,\cdot)|^p \right)(x)\; |u|^{p-2}u(s,x)\;
		ds.
		\endaligned
	\end{equation*}
	
	We have
	
	\begin{lemma}[\cite{Caz:book, tao:book}]\label{Striestimate} Consider
		\begin{equation}\label{lin}
			\left\{
			\aligned i\partial_t u + \Delta u =& \; f,\quad x\in \R^N,\; t\in [0, T),\\
			u(0)=&\; u_0 \in \dot H^1(\R^N),
			\endaligned\right.
		\end{equation}
		where $ \nabla  f \in N(0, T)$, then for some constant $C>0$, we have
		\begin{equation*}
			\aligned \sup_{t\in [0,T)}\big\| u\big\|_{\dot H^1} +
			\big\|u\big\|_{l(0,T)}   \leq C \left( \big\|u_0\big\|_{\dot H^1} +
			\big\|\nabla f\big\|_{N(0,T)}\right).
			\endaligned
		\end{equation*}
	\end{lemma}
	
	\begin{lemma}[\cite{Gra04:book, Ste:70:book}]\label{L:hardy}
		For $\alpha \in (0, N)$, then  there exists a constant $C(N,\alpha)
		$ such that for any $r\in (\frac{N}{N-\alpha}, \infty)$,
		\begin{equation*}
			\left\|\int_{\R^N}\frac{f(y)}{|x-y|^{N-\alpha}} \; dy
			\right\|_{L^r(\R^N)} \leq C(N,\alpha)
			\big\|f\big\|_{L^{\frac{Nr}{N+\alpha r}}(\R^N)}.
		\end{equation*}
	\end{lemma}
	
	\begin{lemma}[\cite{Bahouri-gerard, GMO, KillipVisan:clay}]\label{L:sobolev}
		For every $f\in\dot H^1(\R^N)$, there is a constant $C$ such that
		\begin{equation*}
			\big\|f\big\|_{L^{\frac{2N}{N-2}}(\R^N)} \leq C
			\big\|f\big\|^{\frac{N-2}{N}}_{\dot H^1(\R^N)}
			\big\|f\big\|^{\frac{2}{N}}_{\dot B^1_{2,\infty}(\R^N)}.
		\end{equation*}
	\end{lemma}

	\begin{theorem}[\cite{Arora: wellposedness, Arora:LWP, Arora:local wellposedness}]\label{T:local}
		For any $u_0 \in \dot H^1(\R^N)$ and $t_0\in \R$, there exists a
		unique maximal-lifespan solution $u :I\times \R^N \rightarrow \C$ to
		\eqref{har} with   $u(t_0)=u_0$. This solution also has the
		following properties:
		
		\begin{enumerate}
			\item[\rm(a)]   $I$ is an open neighborhood of $t_0$.
			\item[\rm(b)]  The energy of the solution $u$ are conserved,
			that is, for all $t\in I$, we have
			\begin{equation*}
				\aligned E(u(t))=E(u_0).
				\endaligned
			\end{equation*}
			\item[\rm(c)]   If $u^{(n)}_0$ is a
			sequence converging to $u_0$ in $\dot H^1 $ and
			$u^{(n)}:I^{(n)}\times \R^N \rightarrow \C$ are the associated
			maximal-lifespan solutions, then $u^{(n)}$ converges locally
			uniformly to $u$.
			\item[\rm(d)]   There exists $\eta_0$, such that if $\big\|u_0\big\|_{\dot H^1}<\eta_0$,
			then $u$ is a global solution.  Indeed, the solution also scatters to $0$
			in $\dot H^1(\R^N)$.
		\end{enumerate}
	\end{theorem}

	\subsection{Properties of ground state.}
	
	From \eqref{gs} and Proposition \ref{SharpConstant}, we have
	\begin{equation}\label{KinEnergy}
		\aligned \big\|\nabla W \big\|^2_{2}=C(N,\lambda)^{-\frac{p}{p-1}}, \quad E(W)=\frac{p-1}{2p}C(N,\lambda)^{-\frac{p}{p-1}}
		.
		\endaligned
	\end{equation}
	
	Using the characterization of $W$ in Proposition \ref{SharpConstant}, the
	refined Sobolev inequality in Lemma \ref{L:sobolev}, the analogue
	concentration compactness principle (profile decomposition in $\dot
	H^1_{rad}$) to the proof of Proposition $3.1$ in
	\cite{MiaoXZ:10:blowup}, we have
	\begin{proposition}\label{P:static stability} Let  $u \in \dot H^1(\R^N)$ be radial and $E(u) = E (W)$.
		Then there exists a function $\varepsilon=\varepsilon (\rho)$, such that
		\begin{equation*}
			\aligned  \inf_{ \theta \in \R,\; \mu>0}  \big\| u_{\theta, \mu} -W
			\big\|_{\dot H^1} \leq \varepsilon (\delta(u)), \quad \lim_{\rho
				\rightarrow 0}\varepsilon (\rho)=0,
			\endaligned
		\end{equation*}
		where $u_{\theta, \mu}(x) = e^{i\theta} \mu^{-\frac{N-2}{2}}u(
		\mu^{-1} x)$, $\displaystyle \delta(u)=\Big| \int_{\R^N} \left(
		\big|\nabla u \big|^2 -\big|\nabla W \big|^2 \right) \; dx \Big| $.
	\end{proposition}
	
	\begin{proof} We follow the argument in Proposition $2.7$ in \cite{MiaoWX:dynamic gHartree} and argue by contradiction. Let $u_n \in \dot H^1$ be any radial sequence and satisfy
		\begin{equation}\label{kinetic:converg}
			E(u_n)=E(W), \; \text{and}\; \Big| \int_{\R^N} \left( \big|\nabla
			u_n \big|^2 -\big|\nabla W \big|^2 \right) \; dx \Big|
			\longrightarrow 0.
		\end{equation}
		Hence, from the profile decomposition in $\dot
		H^1_{rad}(\R^N)$, there exists a subsequence of $\{u_n\}^{\infty}_{n=1}$ and a
		sequence $\{U^{(j)}\}_{j\geq 1} $ in $\dot H^1_{rad}$ and for any
		$j\geq 1$, a family $\{\lambda^j_n\}$ such that
		
		\begin{enumerate}
			\item If $j\not= k$, then
			\begin{align*}
				\frac{\lambda^j_n}{\lambda^k_n} + \frac{\lambda^k_n}{\lambda^j_n}
				\longrightarrow +\infty, \text{}\; n\rightarrow +\infty
			\end{align*}
			\item For every $l\geq 1$, we have
			\begin{align*}
				u_n(x) = \sum^l_{j=1} \frac{1}{\left(  \lambda^j_n
					\right)^{(N-2)/2}} U^{(j)} \left( \frac{x}{\lambda^j_n}\right) +
				r^l_n(x), \; \text{with}\; \lim_{l\rightarrow
					+\infty}\limsup_{n\rightarrow+\infty} \big\|r^l_n\big\|_{\dot{B}^1_{2,\infty}} \rightarrow 0.
			\end{align*}
			\item We have
			\begin{align*}
				\big\|\nabla u_n\big\|^2_{L^2} = \sum^{l}_{j=1} \big\|\nabla
				U^{(j)}\big\|^2_{L^2} + \big\|\nabla r^l_n\big\|^2_{L^2} + o_n(1),
			\end{align*}
		\end{enumerate}
		By Lemma \ref{L:hardy}, Lemma \ref{L:sobolev}, the orthogonality
		between $\lambda^j_n$ and $\lambda^k_n$ for $j\not = k$, we have
		\begin{align*}	
			\lim_{l\rightarrow+\infty}\limsup_{n\rightarrow +\infty}
			\iint_{\R^N\times\R^N}
			I_{\lambda}(x-y)\left|r^l_n(x)\right|^p \left|r^l_n(y)\right|^p \;
			dxdy = 0,
		\end{align*}
		\begin{align*}
			\iint_{\R^N\times\R^N} I_{\lambda}(x-y)|u_n(x)|^p |u_n(y)|^p \; dxdy
			=\lim_{l\rightarrow+\infty}\sum^l_{j=1} \iint_{\R^N\times\R^N}
			I_{\lambda}(x-y)\left| U^{(j)}  (x)\right|^p\left| U^{(j)}
			(y)\right|^p \; dxdy .
		\end{align*}
		By Proposition \ref{SharpConstant}, we have
		\begin{align*}
			C(N,\lambda)^{-p} \iint_{\R^N\times\R^N}
			I_{\lambda}(x-y)\left| U^{(j)}  (x)\right|^p\left| U^{(j)}
			(y)\right|^p \; dxdy  \leq
			\big\|\nabla U^{(j)}\big\|^{2p}_{L^2},
		\end{align*}
		thus,
		\begin{align*}
			C(N,\lambda)^{-p} \sum^{l}_{j=1}\iint_{\R^N\times\R^N}
			I_{\lambda}(x-y)\left| U^{(j)}  (x)\right|^p\left| U^{(j)}
			(y)\right|^p \; dxdy  \leq
			\sum^{l}_{j=1}\big\|\nabla U^{(j)}\big\|^{2p}_{L^2}.
		\end{align*}
		This yields that
		\begin{align*}
			C(N,\lambda)^{-p}  \leq &
			\lim_{n\rightarrow+\infty}\limsup_{l\rightarrow+\infty}\frac{
				\displaystyle \sum^{l}_{j=1} \big\|\nabla
				U^{(j)}\big\|^{2p}_{L^2}}{\displaystyle \sum^{l}_{j=1}\iint_{\R^N\times\R^N}
				I_{\lambda}(x-y)\left| U^{(j)}  (x)\right|^p\left| U^{(j)}
				(y)\right|^p \; dxdy} \\
			\leq  & \lim_{n\rightarrow+\infty}\limsup_{l\rightarrow+\infty}
			\frac{ \displaystyle \left(\sum^{l}_{j=1} \big\|\nabla
				U^{(j)}\big\|^2_{L^2}\right)^p}{\displaystyle \sum^{l}_{j=1}\iint_{\R^N\times\R^N}
				I_{\lambda}(x-y)\left| U^{(j)}  (x)\right|^p\left| U^{(j)}
				(y)\right|^p \; dxdy}\\
			\leq & \lim_{n\rightarrow+\infty} \frac{ \displaystyle \big\|\nabla
				u_n \big\|^{2p}_{L^2}} {\displaystyle \iint_{\R^N\times\R^N}
				I_{\lambda}(x-y)\left| u_n  (x)\right|^p\left| u_n
				(y)\right|^p \; dxdy} \\
			= & C(N,\lambda)^{-p},
		\end{align*}
		where in the last step we use \eqref{kinetic:converg}. Therefore, we
		conclude that only one term $U^{(j)}$ is nonzero, and
		\begin{align*}
			C(N,\lambda)^{-p}  = \frac{ \displaystyle \big\|\nabla
				U^{(j_0)}\big\|^{2p}_{L^2}}{\displaystyle \iint_{\R^N\times\R^N}
				I_{\lambda}(x-y)\left| U^{(j_0)}  (x)\right|^p\left| U^{(j_0)}
				(y)\right|^p \; dxdy }.
		\end{align*}
		By Proposition \ref{SharpConstant} for the radial case, there exist $\theta_0\in [0, 2\pi)$,
		$\lambda_0>0$ such that
		\begin{equation*}
			U^{(j_0)}(x) = e^{i\theta_0} \lambda^{-\frac{N-2}{2}}_0W (
			\lambda^{-1}_0 x ).
		\end{equation*}
		This contradicts with the assumption and completes the proof.
	\end{proof}
	
	\subsection{The coercive property.}
	By the convex analysis in \cite{MiaoXZ:09:e-critical radial Har}, we  have
	\begin{align}
		\label{convexity} \forall u \in \dot H^1 , \;\; E(u)\leq E(W), \;\;
		\big\|\nabla u \big\|_{L^2} \leq \big\| \nabla W \big\|_{L^2}
		\Longrightarrow \frac{\big\|\nabla u \big\|^2_{L^2}}{\big\|\nabla W
			\big\|^2_{L^2}} \leq \frac{E(u)}{E(W)},\\
		\forall u \in \dot H^1 , \;\; E(u)\leq E(W), \;\; \big\|\nabla u
		\big\|_{L^2} \geq \big\| \nabla W \big\|_{L^2} \Longrightarrow
		\frac{\big\|\nabla u \big\|^2_{L^2}}{\big\|\nabla W \big\|^2_{L^2}} \geq
		\frac{E(u)}{E(W)}. \notag
	\end{align}
	This, together with  the energy conservation, the variational
	characterization of $W$ and the continuity argument, implies that
	\begin{lemma}\label{energytrapp}
		Let $u \in C(I,  \dot H^1(\R^N))$ be a radial solution of \eqref{har} with
		initial data $u_0$, and $I=(-T_-, T_+)$ its maximal interval of
		existence. Assume that $ E(u_0)=E(W),$ then
		\begin{enumerate}
			\item[\rm (a)] if $\big\|\nabla u_0\big\|_{2}< \big\|\nabla W\big\|_{2}$,
			then $\big\|\nabla u(t)\big\|_{2}< \big\|\nabla W\big\|_{2}$ for
			$t\in I$.
			\item[\rm (b)]  if $\big\|\nabla u_0\big\|_{2} = \big\|\nabla W\big\|_{2}$,
			then $u =W $ up to the symmetry of the equation.
			\item[\rm (c)] if $\big\|\nabla u_0\big\|_{2} >  \big\|\nabla W\big\|_{2}$,
			then $\big\|\nabla u(t)\big\|_{2} > \big\|\nabla W\big\|_{2}$ for $
			t\in I$.
		\end{enumerate}
	\end{lemma}
	\begin{proof}The proof is analogue to Lemma $2.8$ in \cite{MiaoWX:dynamic gHartree} and Proposition 3.1 in \cite{MiaoXZ:09:e-critical radial Har}.
	\end{proof}
	
	\subsection{Monotonicity formula}
	We define the localized Virial term
	\begin{equation}\label{localV}
		\aligned V_R(t)=\int_{\R^N} \phi_R(x)\big|u(t,x)\big|^2 dx,\;
		\endaligned
	\end{equation}
	where $\phi(x)$ be a positive smooth radial function such that 
	$\phi(x)=\frac{|x|^2}{2}$ for $|x|\leq 1$, $\phi(x)=0$ for $|x|\geq 2$, $\phi''(r) \leq 1,\; r\geq 0$ and $\phi_R(x)=R^2\phi\left(\frac{x}{R}\right),R>0$, which is useful to analyze the gradient variant from W in Section \ref{S:modulation}. By simple calculations as in  \cite{MiaoWX:dynamic gHartree, MiaoXZ:09:e-critical radial Har, T. S and C. Xu: virial identity}, we have 
	
	\begin{lemma}\label{L:local virial}Let $u(t,x)$ be a radial solution to \eqref{har}, $V_R(t)$ be
		defined by \eqref{localV}, then
		\begin{align*}
			\partial_t V_R(t)=&\;  2\Im \int_{\R^N} \overline{u}\; \nabla u \cdot \nabla
			\phi_R\; dx, \\
			\partial^2_t V_R(t)=&\; 4 \int_{\R^N} \Big|\nabla u (t,x)\Big|^2 dx -4
			\int_{\R^N}  \left(\ I_{\lambda}*|u|^p\right)\; |u(x)|^p \; dx
			+ A_R\big(u(t)\big),
		\end{align*}
		where
		\begin{align*} A_R\big(u(t)\big)=&\;    \int_{\R^N} \left( 4\phi''\Big(\frac{|x| }{R }\Big)-4 \right) \big|\nabla u (t,x) \big|^2 dx
			+  \int_{\R^N} \big(-\Delta \Delta \phi_R(x) \big) \big| u(t,x)\big|^2 dx \\
			& + O\left(\int_{|x|\geq R}  \left(I_{\lambda}*|u|^p\right)\; |u(x)|^p \; dx\right).
		\end{align*}
	\end{lemma}

	\subsection{Basic properties of the linearized operator.} We consider a radial solution $u$ of \eqref{har} close to $W$ and
	write $u$ as
	\begin{equation*}
		\aligned u(t,x)= W(x)+h(t,x),
		\endaligned
	\end{equation*}
	then $h$ satisfies the difference equation
	\begin{align}
		i\partial_t h+\Delta h= V h+  i R(h ), \nonumber
	\end{align}
	where the linear operator $V$ and the remainder $R(h)$ are defined
	by
	\begin{align}
		\label{linearterm} Vh: = &\; -(p-1) (I_{\lambda}*W^p) W^{p-2}h_1- p[I_{\lambda}*(W^{p-1}h_1)] W^{p-1}- i(I_{\lambda}*W^p)W^{p-2} h_2,\\
		R(h):=& i \Big((I_{\lambda}*|W+h|^p) |W+h|^{p-2}(W+h)- (I_{\lambda}*W^p)W^{p-2}W -(p-1) (I_{\lambda}*W^p) W^{p-2}h_1 \notag \\
		&- p[I_{\lambda}*(W^{p-1}h_1)] W^{p-1}- i(I_{\lambda}*W^p)W^{p-2} h_2\Big) .  \label{remainder}
	\end{align}
	
	Let $h_1=\Re h$ and $h_2=\Im h$ (In the following context, we will
	always denote the complex value function $h=h_1+ih_2=(h_1, h_2)$
	without confusion). Then $h$ is a solution of the equation
	\begin{equation}\label{linearequat}
		\aligned
		\partial_t h + \mathcal{L} h= -i R(h), \quad \mathcal{L}:=\begin{pmatrix} 0 &  -L_-\\
			L_+  & 0 \end{pmatrix},
		\endaligned
	\end{equation}
	where the self-adjoint operators $L_{\pm}$ are defined by
	\begin{align}
		L_+ h_1:=& -\Delta h_1 -(p-1)(I_{\lambda}*W^p)W^{p-2} h_1 -  p[I_{\lambda}*(W^{p-1}h_1)] W^{p-1}, \label{Lpos}\\
		L_- h_2:= &-\Delta h_2-(I_{\lambda}*W^p)W^{p-2} h_2.
		\label{Lneg}
	\end{align}
	
	Now we first give some basic estimates about the linearized
	equation \eqref{linearequat}. Firstly, we give nonlinear estimates as follows.
	
	\begin{lemma}\label{linearoperator:prelimestimate}
		Let $V$, $R$ be defined by \eqref{linearterm} and \eqref{remainder},
		respectively, $I$ be a interval with $|I|\leq 1$, and $g,h \in
		l(I)$, $u, v \in L^{\frac{2N}{N-2}}(\R^N)$, then
		\begin{align}
			\big\|\nabla \big(V h\big) \big\|_{N(I)} \lesssim &\; |I|^{\frac{p-1}{p}}
			\big\|h\big\|_{l(I)},\label{linearestimate}
			\\
			\big\|\nabla \big( R(g)-R(h)\big)
			\big\|_{N(I)}\lesssim & \; \big\|g-h\big\|_{l(I)} \Big(
			|I|^{\frac{2p-3}{2p}} \big(\big\|g\big\|_{l(I)} + \big\|h\big\|_{l(I)}\big)
			+ \big\|g\big\|^{2(p-1)}_{l(I)} +
			\big\|h\big\|^{2(p-1)}_{l(I)}\Big),\label{nonlinearestimates}\\
			\big\|  R(u)-R(v) \big\|_{L^{\frac{2N}{N+2}}(\R^N)}\lesssim & \;
			\Big(\big\|u\big\|_{L^{\frac{2N}{N-2}}} +
			\big\|v\big\|_{L^{\frac{2N}{N-2}}}+
			\big\|u\big\|^{2(p-1)}_{L^{\frac{2N}{N-2}}} +
			\big\|v\big\|^{2(p-1)}_{L^{\frac{2N}{N-2}}}\Big)\big\|u-v\big\|_{L^{\frac{2N}{N-2}}}.
			\label{nonlinearestimate:dual}
		\end{align}
	\end{lemma}
	
	\begin{remark}\label{calculation of nonlinearity term}
		Let $f(z)=|z|^{p-2}z$. The complex derivative of $f$ is that
		$$f_z(z)=\frac{p}{2}|z|^{p-2}, \quad f_{\bar{z}}(z)=\frac{p-2}{2}|z|^{p-4}z^2.$$
		For $z_1,z_2\in\mathbb{C}$, we have
		$$f(z_1)-f(z_2)=\int_{0}^{1}[f_{z_1}(z_2+\theta(z_1-z_2))(z_1-z_2)+f_{\bar{z_1}}(z_2+\theta(z_1-z_2))\overline{(z_1-z_2)}]d\theta.$$
		Therefore,
		$$f(z_1)-f(z_2)\lesssim(|z_1|^{p-2}+|z_2|^{p-2})|z_1-z_2|.$$
		Also, by \cite{Cazenave:fractional order}, we get
		\begin{equation*}
			\big||z_1|^{\alpha}-|z_2|^{\alpha}\big|\lesssim\left\{\begin{aligned}
				&(|z_1|^{{\alpha}-1}+|z_2|^{{\alpha}-1})|z_1-z_2|&\alpha\geq1;\\
				&|z_1-z_2|^{\alpha}&0<\alpha\leq1.
			\end{aligned}\right.
		\end{equation*}
	\end{remark}
	
	\begin{proof}[Proof of Lemma \ref{linearoperator:prelimestimate}] We follow the argument in Lemma $2.10$ in \cite{MiaoWX:dynamic gHartree}. We firstly prove \eqref{linearestimate}. Recall that
		$$ Vh: = \; -(p-1) (I_{\lambda}*W^p) W^{p-2}h_1- p[I_{\lambda}*(W^{p-1}h_1)] W^{p-1}- i(I_{\lambda}*W^p)W^{p-2} h_2.$$
		Hence, we get
		\begin{align}\label{Vh}
			\big|\nabla \big(V h\big) \big| \lesssim &\;|(I_{\lambda}*W^p)W^{p-3}\nabla Wh|+|(I_{\lambda}*W^p)W^{p-2} \nabla h|+|(I_{\lambda}*(W^{p-1}\nabla W))W^{p-2}h|\nonumber\\
			&\;+|(I_{\lambda}*(W^{p-1}h_1))W^{p-2}\nabla W|+|I_{\lambda}*(W^{p-2}\nabla W h_1)W^{p-1}|\nonumber\\
			&\;+|(I_{\lambda}*(W^{p-1}\nabla h_1)W^{p-1}|.
		\end{align}
		As for the first term on the right of \eqref{Vh}, by Lemma \ref{L:hardy} , Remark \ref{calculation of nonlinearity term} and
		H\"{o}lder's inequality, we have
		\begin{align*}
			\big\|(I_{\lambda}*W^p)W^{p-3}\nabla Wh\big\|_{N(I)}\lesssim& \big\|\nabla W \big\|_{S(I)}\big\|W\big\|^{p-3}_{Z(I)}\big\|h\big\|_{Z(I)}\big\|I_{\lambda}*W^p\big\|_{L_t^2L_x^{\frac{2N}{2N+2p-Np-2}}}\\
			\lesssim& \big\|\nabla W \big\|_{S(I)}\big\|W\big\|^{p-3}_{Z(I)}\big\|h\big\|_{Z(I)}\big\|W^p\big\|_{L_t^2L_x^{\frac{2N}{Np-2p-2}}}\\
			=&\big\|W\big\|^{2p-3}_{Z(I)}\big\|\nabla W \big\|_{S(I)}
			\big\|h\big\|_{Z(I)}.
		\end{align*}
		As for the rest term on the right of \eqref{Vh}, by simple calculations, we have
		\begin{align*}
			\big\|\nabla \big( V h \big) \big\|_{N(I)}
			\lesssim \big\|W\big\|^{2(p-1)}_{Z(I)} \big\|\nabla h\big\|_{S(I)} +
			\big\|W\big\|^{2p-3}_{Z(I)}\big\|\nabla W \big\|_{S(I)}
			\big\|h\big\|_{Z(I)}.
		\end{align*}
		Since $W \in L^{\frac{2Np}{Np-2p-2}}_x \cap \dot W^{1,
			\frac{2Np}{Np-2}}_x,$ we have that
		\begin{equation*}
			\aligned \big\|W\big\|_{Z(I)} \lesssim |I|^{\frac{1}{2p}}, \;\;
			\big\|\nabla W \big\|_{S(I)}\lesssim |I|^{\frac{1}{2p}},
			\endaligned
		\end{equation*}
		which implies \eqref{linearestimate}. 
		
		As for  \eqref{nonlinearestimates} and \eqref{nonlinearestimate:dual}, recall that
		\begin{align*}
			R(h):=&\; i \Big((I_{\lambda}*|W+h|^p) |W+h|^{p-2}(W+h)- (I_{\lambda}*W^p)W^{p-1} -(p-1) (I_{\lambda}*W^p) W^{p-2}h_1\\
			&- p[I_{\lambda}*(W^{p-1}h_1)] W^{p-1}- i(I_{\lambda}*W^p)W^{p-2} h_2\Big) ,
		\end{align*}
		so, we have
		\begin{align*}
			R(u)-R(v)&=\; i \Big((I_{\lambda}*|W+u|^p) |W+u|^{p-2}(W+u)- (I_{\lambda}*|W+v|^p)|W+v|^{p-2} (W+v)\nonumber\\
			&\qquad-(p-1)(I_{\lambda}*W^p)W^{p-2}(u_1-v_1) - p\left(I_{\lambda}*(W^{p-1}(u_1-v_1))\right) W^{p-1}  \\
			& \qquad-i(I_{\lambda}*W^p)W^{p-2}(u_2-v_2)  \Big). 
		\end{align*}
		As for the term $(I_{\lambda}*|u|^p) |u|^{p-2}u- (I_{\lambda}*|v|^p)|v|^{p-2}v$, by Lemma \ref{L:hardy} , Remark \ref{calculation of nonlinearity term} and H\"{o}lder's inequality, we estimate when $p>3$ that
		
		\begin{align*}
			&\big\|(I_{\lambda}*|u|^p)|u|^{p-2} u- (I_{\lambda}*|v|^p)|v|^{p-2}v\big\|_{L^{\frac{2N}{N+2}}}\\
			=&\big\|(I_{\lambda}*|u|^p)|u|^{p-2}(u-v)\big\|_{L^{\frac{2N}{N+2}}}+\big\|(I_{\lambda}*(|u|^p-|v|^p))|u|^{p-2}v\big\|_{L^{\frac{2N}{N+2}}}\\
			&+\big\|(I_{\lambda}*|v|^p)(|u|^{p-2}-|v|^{p-2})v\big\|_{L^{\frac{2N}{N+2}}}\\
			\lesssim& \big\|u\big\|^{p-2}_{L^{\frac{2N}{N-2}}}\big\|u-v\big\|_{L^{\frac{2N}{N-2}}}\big\|I_{\lambda}*|u|^p\big\|_{L^{\frac{2N}{2N+2p-Np}}}+\big\|u\big\|^{p-2}_{L^{\frac{2N}{N-2}}}\big\|v\big\|_{L^{\frac{2N}{N-2}}}\big\|I_{\lambda}*(|u|^{p}-|v|^{p})\big\|_{L^{\frac{2N}{2N+2p-Np}}}\\
			&+\big\||u|^{p-2}-|v|^{p-2}\big\|_{L^{\frac{2N}{(N-2)(p-2)}}}\big\|v\big\|_{L^{\frac{2N}{N-2}}}\big\|I_{\lambda}*|v|^p\big\|_{L^{\frac{2N}{2N+2p-Np}}}\\
			\lesssim& \big\|u\big\|^{p-2}_{L^{\frac{2N}{N-2}}}\big\|u-v\big\|_{L^{\frac{2N}{N-2}}}\big\|u^p\big\|_{L^{\frac{2N}{(N-2)p}}}+\big\|u\big\|^{p-2}_{L^{\frac{2N}{N-2}}}\big\|v\big\|_{L^{\frac{2N}{N-2}}}\big\||u|^{p}-|v|^{p}\big\|_{L^{\frac{2N}{(N-2)p}}}\\
			&+\big\||u|^{p-2}-|v|^{p-2}\big\|_{L^{\frac{2N}{(N-2)(p-2)}}}\big\|v\big\|_{L^{\frac{2N}{N-2}}}\big\||v|^p\big\|_{L^{\frac{2N}{(N-2)p}}}\\
			\lesssim& \big\|u\big\|^{2p-2}_{L^{\frac{2N}{N-2}}}\big\|u-v\big\|_{L^{\frac{2N}{N-2}}}+\big\|u\big\|^{p-2}_{L^{\frac{2N}{N-2}}}\big\|v\big\|_{L^{\frac{2N}{N-2}}}\big\|(|u|^{p-1}+|v|^{p-1})|u-v|\big\|_{L^{\frac{2N}{(N-2)p}}}\\
			&+\big\|(|u|^{p-3}-|v|^{p-3})|u-v|\big\|_{L^{\frac{2N}{(N-2)(p-2)}}}\big\|v\big\|_{L^{\frac{2N}{N-2}}}\big\||v|^p\big\|_{L^{\frac{2N}{(N-2)p}}}\\
			\lesssim& \big\|u\big\|^{2p-2}_{L^{\frac{2N}{N-2}}}\big\|u-v\big\|_{L^{\frac{2N}{N-2}}}+\big\|u\big\|^{2p-3}_{L^{\frac{2N}{N-2}}}\big\|v\big\|_{L^{\frac{2N}{N-2}}}\big\|u-v\big\|_{L^{\frac{2N}{N-2}}}+\big\|u\big\|^{p-2}_{L^{\frac{2N}{N-2}}}\big\|u-v\big\|_{L^{\frac{2N}{N-2}}}\big\|v\big\|^{p}_{L^{\frac{2N}{N-2}}}\\
			&+\big\|u-v\big\|_{L^{\frac{2N}{N-2}}}\big(\big\|u\big\|^{p-3}_{L^{\frac{2N}{N-2}}} \big\|v\big\|^{p+1}_{L^{\frac{2N}{N-2}}}+\big\|v\big\|^{2p-2}_{L^{\frac{2N}{N-2}}}\big)\\
			\lesssim& 
			\big\|u-v\big\|_{L^{\frac{2N}{N-2}}}\big(\big\|u\big\|_{L^{\frac{2N}{N-2}}} +
			\big\|v\big\|_{L^{\frac{2N}{N-2}}}+\big\|u\big\|^{2(p-1)}_{L^{\frac{2N}{N-2}}}+\big\|v\big\|^{2(p-1)}_{L^{\frac{2N}{N-2}}}\big).
		\end{align*}
		When $2\leq p\leq 3$ and the other terms can also be estimated similarly. Above all, we can obtain \eqref{nonlinearestimate:dual}.
		
		Last, 
		we can show \eqref{nonlinearestimates} from the expression of $R(u)-R(v)$ above, and completes the proof. 
	\end{proof}

	Because of the linear operator $V$ in the
	linearized equation \eqref{linearequat}, we will use the
	following integral summation argument.
	
	\begin{lemma}[\cite{DuyMerle:NLS:ThresholdSolution}]
		\label{summation} Let $t_0>0$, $p\in [1,+\infty[$, $a_0 \neq 0$, $E$ be
		a normed vector space, and $f\in L^p_{\rm loc}(t_0,+\infty;E)$ such
		that
		\begin{equation}
			\label{small.tau} \exists\; \tau_0>0,\; \exists\; C_0>0, \;\forall\;
			t\geq t_0, \quad \|f\|_{L^p(t,t+\tau_0,E)}\leq C_0e^{a_0 t}.
		\end{equation}
		Then for $t\geq t_0$,
		\begin{align}
			\|f\|_{L^p(t,+\infty,E)}\leq &\;
			\frac{C_0e^{a_0 t}}{1-e^{a_0\tau_0}},\;  \text{if
			}\; a_0<0; \label{conclu.summation} \\
			\|f\|_{L^p(t_0, t,E)} \leq
			&\; \frac{C_0e^{a_0 t}}{1-e^{-a_0\tau_0}},\;  \text{if }\; a_0>0.
			\label{conclu.summationless}
		\end{align}
	\end{lemma}
	
	By the Strichartz estimate, Lemma
	\ref{linearoperator:prelimestimate} and Lemma \ref{summation}, we
	have
	\begin{lemma}\label{linearstrich}
		Let $C$, $c>0$, and $v$ be a solution of \eqref{linearequat} with
		\begin{equation*}
			\big\| v(t) \big\|_{\dot H^1} \leq C e^{-c t},
		\end{equation*}
		then for any $(q,r)$ with $\frac2q=N
		(\frac12 -\frac1r ) $, $q\in [2, +\infty]$, we have for large $t$
		\begin{align*}
			\big\|v\big\|_{l(t,+\infty)}+\big\|\nabla v\big\|_{L^q(t,+\infty;
				L^r)} \leq C e^{-ct}.
		\end{align*}
	\end{lemma}
	
	\begin{proof}We follow the argument in Lemma $2.12$ in \cite{MiaoWX:dynamic gHartree}. For small $\tau_0$, by the Strichartz estimate and
		Lemma \ref{linearoperator:prelimestimate}, we have on $I=[t,
		t+\tau_0]$
		\begin{align*}
			\big\|v\big\|_{l(I)}+\big\|\nabla v\big\|_{L^q(I; L^r)} \leq & C
			e^{-c t } + C\big\|\nabla (V v) \big\|_{N(I)} +  C\big\|\nabla R(
			v) \big\|_{N(I)} \\
			\leq & C e^{-c t } + C |I|^{\frac{p-1}{p}} \big\|v\big\|_{l(I)} +
			\big\|v\big\|^2_{l(I)} + \big\|v\big\|^{2p-1}_{l(I)}.
		\end{align*}
		By choosing sufficiently small $\tau_0$, the continuous argument
		gives that
		\begin{equation*}
			\big\|v\big\|_{l(I)}+\big\|\nabla v\big\|_{L^q(I; L^r)} \leq C
			e^{-c t }.
		\end{equation*}
		Thus this implies the desired result by Lemma \ref{summation}.
	\end{proof}
	
	\subsection{Spectral properties of the linearized operator.}
	Because of the phase rotation invariance and the scaling invariance
	of \eqref{har}, we know that $iW \in \dot{H}^1(\R^N)$ and
	$\widetilde{W}=\dfrac{N-2}{2}W+x\cdot \nabla W \in \dot{H}^1(\R^N)$ are
	elements of the null-space of $\mathcal{L}$ in $\dot{H}^1_{rad}$. In fact,
	they are the only elements of the null-space of $\mathcal{L}$ in
	$\dot{H}^1_{rad}$. From Proposition \ref{prop:nondegeneracy}, we have
	\begin{lemma}\label{keyassumption}Let $\mathcal{L}$ be defined by
		\eqref{linearequat}. Then
		\begin{align*}
			\big\{u\in   \dot{H}^1_{rad}(\R^N),
			\mathcal{L}u=0\big\}=\text{span}\big\{iW, \widetilde{W}
			\big\}.
		\end{align*}Namely,
		\begin{equation*}
			\big\{u\in   \dot{H}^1_{rad}(\R^N), L_-u=0\big\}=\text{span}\big\{W
			\big\};\quad \big\{u\in   \dot{H}^1_{rad}(\R^N),  L_+u=0\big\}=\text{span}\big\{\widetilde{W}
			\big\}.
		\end{equation*}
	\end{lemma}
	
	Now we define the linearized energy $\Phi$ by
	\begin{align}\label{linearized energy}
		\Phi(h):=& \frac12 \int_{\R^N} (L_+ h_1) h_1 \; dx + \frac12
		\int_{\R^N}
		(L_- h_2) h_2 \; dx \\
		= & \frac12 \int_{\R^N} \big| \nabla h_1 \big| ^2 - \frac{p-1}{2}
		\int_{\R^N} (I_{\lambda}*W^p)W^{p-2} h_1^2\; dx-
		\frac{p}{2}
		\int_{\R^N} (I_{\lambda}*W^{p-1}h_1)W^{p-1} h_1\; dx \nonumber\\
		& + \frac12 \int_{\R^N} \big| \nabla h_2 \big| ^2 - \frac12
		\int_{\R^N} (I_{\lambda}*W^p)W^{p-2} h_2^2\; dx,
		\nonumber
	\end{align}
	and  denote by $B(g,h)$ the bilinear symmetric form associated to
	$\Phi$,
	\begin{equation}\label{bilinear form}
		\aligned B(g,h):=\frac12 \int_{\R^N} (L_+ g_1) h_1 \; dx +
		\frac12\int_{\R^N} (L_- g_2) h_2 \; dx, \;\; \forall \; g,h \in \dot
		H^1(\R^N).
		\endaligned
	\end{equation}
	
	Let us specify the important coercivity properties of $\Phi$.
	Consider the three orthogonal directions $W, iW,$ and
	$\widetilde{W}$ in the Hilbert space $\dot H^1(\R^N, \C)$. Let
	$H:=\text{span} \{W, iW, \widetilde{W}\},$ and $H^{\bot}$ be its
	orthogonal subspace in $\dot H^1_{\text{rad}} (\R^N)$, that is,
	\begin{equation*}\label{H orth}
		\aligned H^{\bot} := \left\{v\in \dot H^1_{\text{rad}} (\R^N), (iW, v)_{\dot
			H^1}=(\widetilde{W}, v)_{\dot H^1}  = (W, v)_{\dot H^1}=0 \right\}.
		\endaligned
	\end{equation*}
	
	Following the arguments in Proposition $2.13$ and Proposition $2.14$ in \cite{MiaoWX:dynamic gHartree}, we have the following coercivity properties of the linearized energy $\Phi$ on $H^{\bot}$ and $G_{\bot}$.
	
	\begin{proposition}\label{coerH}There exists a constant $c>0$ such
		that for all radial function $ h \in H^{\bot}$, we have
		\begin{equation}
			\aligned  \Phi(h)\geq c\big\| h \big\|^2_{\dot H^1}.
			\endaligned
		\end{equation}
	\end{proposition}
	For the sake of completeness, we will show it in Appendix \ref{appen CoerH}.

	\begin{proposition}\label{spectral}Let $\sigma(\mathcal{L})$ be the spectrum of the operator
		$\mathcal{L}$, defined on $L^2(\R^N)$ with domain $ H^2(\R^N)$ and let
		$\sigma_{ess}(\mathcal{L})$ be its essential spectrum. Then we have
		\begin{enumerate}
			\item[\rm (a)]The operator $\mathcal{L}$ admits two radial
			eigenfunctions $\mathcal{Y}_{\pm} \in \mathcal{S}(\R^N) $ with real
			eigenvalues $\pm e_0$, $e_0
			> 0$, that is, $ \mathcal{L} \mathcal{Y}_{\pm} = \pm e_0
			\mathcal{Y}_{\pm},\; \mathcal{Y}_+ = \overline{\mathcal{Y}}_- ,\;
			e_0
			> 0.
			$
			
			\item[\rm (b)] There exists a constant $c>0$ such that  for all radial function $ h
			\in G_{\bot}$, we have
			\begin{equation*}
				\aligned  \Phi(h)\geq c\big\| h \big\|^2_{\dot H^1},
				\endaligned
			\end{equation*}
			where \begin{equation*}
				\label{G orthy} G_{\bot}  =\left\{v\in \dot H^1, (iW, v)_{\dot
					H^1}=(\widetilde{W}, v)_{\dot H^1}  = B(\mathcal{Y}_\pm, v)  =
				0\right\}.
			\end{equation*}
			
			\item[\rm (c)]
			$\sigma_{ess}(\mathcal{L})=\{i\xi: \xi\in \R, \}, \quad
			\sigma(\mathcal{L})\cap \R =\{-e_0, 0, e_0\}. $
		\end{enumerate}
	\end{proposition}
	\begin{remark}
		The proofs of (b) and (c)  of Proposition \ref{spectral} are analogue to that of Proposition $2.14$ in \cite{MiaoWX:dynamic gHartree}, but there is a little difference in the proof of (a). Here, we consider the case $N\geq3$ rather than the case $N>4=\lambda$ in \cite{MiaoWX:dynamic gHartree}. However, when $N=3,4,\;W\notin L^2(\R^N)$,  we will give the proof of (a) for the completeness in Appendix \ref{appen-spectralprop}. 
	\end{remark}

	\begin{remark}\label{r:bilinear form prop}
		As a consequence of the definition of $\Phi$ and $B$, we have that for
		any $ f, g \in \dot H^1(\R^N), \mathcal{L}f, \mathcal{L}g \in \dot H^1 (\R^N)$,
		and $ h\in H^{\bot}$
		\begin{align}
			\Phi(iW)=&\Phi(\widetilde{W})=\Phi(\mathcal{Y}_\pm)=0, \;\; \Phi(W)=-\big\|\nabla W\big\|^2_2<0,  \label{Phi}\\
			B(f,g)=&B(g,f), \;\; B(\mathcal{L}f, g)=-B(f,\mathcal{L}g), \label{B-pro}\\
			B(iW, f)=&B(\widetilde{W}, f)=0, \;\;  B(W, h)=0
			.\label{B-orthEnerHbot}
		\end{align}
	\end{remark}
	
	\begin{corollary}
		As a consequence of  Proposition \ref{coerH}, we have
		\begin{align}
			B(\mathcal{Y}_+, \mathcal{Y}_-)\not = 0. \label{B-orthEigen}
		\end{align}
	\end{corollary}
	\begin{proof}  If $B(\mathcal{Y}_+, \mathcal{Y}_-) = 0$, then for any $h\in \text{span} \{iW, \widetilde{W}, \mathcal{Y}_{\pm}\}$, which is of dimension 4, we have
		$$\Phi(h)=0.$$
		But by Proposition \ref{coerH}, we know that $\Phi$ is positive on
		$H^{\bot} \cap \dot H^1_{\text{rad}}$, which is of codimension 3. It is a
		contradiction.
	\end{proof}

	\begin{corollary}As a consequence of Proposition
		\ref{spectral}, we have
		\begin{equation}\label{WYrelation}
			\aligned \int \nabla W \cdot \nabla \mathcal{Y}_1 \; dx \not =0,
			\quad \text{where}\; \mathcal{Y}_1=\Re \mathcal{Y}_+.
			\endaligned
		\end{equation}
	\end{corollary}
	\begin{proof}
		Note that
		$$(W, iW)_{\dot H^1}=(W, \widetilde{W})_{\dot H^1}=0.$$
		By \eqref{har}, we have $ L_+(W)= 2(p-1) \Delta W.$ If $ (W,
		\mathcal{Y}_1)_{\dot H^1} =0,$ then by \eqref{B-pro}, we obtain
		\begin{equation*}
			\aligned \pm e_0 B(\mathcal{Y}_{\pm}, W) = &  B(\mathcal{L}
			\mathcal{Y}_{\pm}, W) = - B(\mathcal{Y}_{\pm}, \mathcal{L} W) \\
			=& \mp\frac12 \int L_{-}\mathcal{Y}_2 \cdot L_+(W) = \pm\frac12 \int e_0
			\mathcal{Y}_1 \cdot  L_+(W)  =  \pm(p-1)\int e_0 \mathcal{Y}_1 \cdot \Delta
			W =0,
			\endaligned
		\end{equation*}
		which means that $W \in G_{\bot}$. By Proposition \ref{spectral}, we
		have $\Phi(W) \gtrsim \big\|W \big\|^2_{\dot H^1}$, which
		contradicts  with \eqref{Phi}.
	\end{proof}

	%
	%
	%
	%
	
	\section{Existence of special threshold solutions $W^{\pm}$}\label{S:existence}
	In this section, we follow the arguments in \cite{DuyMerle:NLS:ThresholdSolution, DuyMerle:NLW:ThresholdSolution, LiZh:NLS, LiZh:NLW} and \cite {MiaoWX:dynamic gHartree} to show the existence of special threshold solutions $U^a$ and $W^{\pm}$ in the energy space, we can also refer to \cite{RaS11:NLS:mini sol} for the ideas from the existence of minimal mass blow-up solutions for mass-critical NLS.  We firstly construct a family of high-order approximate
	solutions to \eqref{har} by use of the spectral property of the linearized operator $\mathcal{L}$ and a series of the elliptic iterations, and then obtain the special solution $U^a$ by the compactness argument. The proof here is  similar to those in Section 3 in \cite{MiaoWX:dynamic gHartree},  we only give the sketch of proof.
	
	\subsection{A family of approximate solutions converging to $W$ as $t\rightarrow +\infty$.}

	\begin{proposition}\label{constru:approximatesolution}
		Let $a\in \R$. There exists a sequence of functions
		$(\ZZZ^a_j)_{j\geq 1}$ in $\mathcal{S}(\R^N)$ such that
		$\ZZZ^a_1=a\mathcal{Y}_+$ and if
		\begin{equation}\label{h app}
			\aligned
			h^a_k(t,x):=  \sum^k_{j=1}e^{-je_0t}\mathcal{Z}^a_j,\; k\in \Z^+,
			\endaligned\end{equation}
		then we have as $t\rightarrow +\infty$
		\begin{equation}\label{approxerror}
			\aligned \epsilon_k:=&  \partial_t  h^a_k + \mathcal{L}  h^a_k  -
			R(h^a_k) = O(e^{-(k+1)e_0t}) \quad \text{in }\;\mathcal{ S }(\R^N).
			\endaligned
		\end{equation}
	\end{proposition}
	
	\begin{remark}\label{epsilon}
		Let  
		\begin{equation}\label{u app}
			\aligned
			U^a_k(t,x):= W(x)+ h^a_k(t,x),
			\endaligned\end{equation}
		then as $t\rightarrow +\infty$, we have the high-order error estimate
		\begin{equation*}
			\aligned \epsilon_k:=& i \partial _t U^a_k + \Delta  U^a_k  + [I_{\lambda}*|U^a_k|^{p}]|U^a_k|^{p-2} U^a_k = O(e^{-(k+1)e_0t}) \quad
			\text{in }\; \mathcal{ S }(\R^N).
			\endaligned
		\end{equation*}
	\end{remark}
	
	\begin{proof}[Sketch of proof]
		Using the eigenfunctions of the real eigenvalues of the linearized operator, we can define $h^a_1:=e^{-e_0t}\ZZZ^a_1$. By the definition of \eqref{remainder}, we know that $R$ has at
		least the quadratic term, that is, $R \left(e^{-e_0t}\ZZZ^a_1\right)
		= O(e^{-2e_0t})$, which yields \eqref{approxerror} for $k=1$. we can use  a series of the elliptic iterations to cancel low-order error term, which together with the induction argument implies the existence of $\mathcal{Z}^a_j$, $j=2,\ldots$.
	\end{proof}
	
	\subsection{Construction of special solutions near an approximate solution.} Now we show the
	existence of special radial threshold solution which can be proved by the fixed point argument.
	\begin{proposition}\label{existence:thresholdsolution}
		Let $a \in \R$. There exist $k_0>0$ and $t_0 \geq 0$ such that for
		any $k\geq k_0$, there exists a radial solution $U^a$ of \eqref{har}
		such that for $t\geq t_0$
		\begin{equation}\label{difference:Stri} \aligned
			\left\|  U^a-U^a_k  \right\|_{l(t, +\infty)}  \leq
			e^{-(k+\frac{1}{2})e_0t}.
			\endaligned
		\end{equation}
		Furthermore $U^a$ is the unique solution of \eqref{har} satisfying
		\eqref{difference:Stri} for large $t$, and $U^a\in L^2(\R^N)$ if $N\geq 5$. Finally, $U^a$ is independent
		of $k$ and satisfies for $t\geq t_0$,
		\begin{equation}\label{difference:Energy}
			\aligned   \big\|  U^a(t)- W - a e^{-e_0t}\mathcal{Y}_+ \big\|_{\dot
				H^1(\R^N)}  \leq e^{-\frac{3}{2}e_0t}.
			\endaligned
		\end{equation}
	\end{proposition}
	
	\begin{proof}[Sketch of proof]
		The proof depends on the fixed point argument and Strichartz estimate in Lemma \ref{Striestimate}. Let $e
		:=U^a-U^a_k=h^a- h^a_k$ and it satisfies
		$$\partial_t e +\LLL
		e = R(h^a_k+e )- R(h^a_k)-\eps_k,$$ 
		where $h^a:=U^a-W,\; h^a_k:=U^a_k-W,\;\eps_k$ is defined in Remark \ref{epsilon}. This equation may be rewritten as
		\begin{align*}
			\label{solve:h}  i\partial_t e+\Delta e =& V e+  i R(h^a_k+e)-i
			R(h^a_k)-i\eps_k.
		\end{align*}
		
		Let us define the mapping
		\begin{equation*}
			\aligned\MMM_k(e)(t):=-\int_t^{+\infty}e^{i(t-s)\Delta} \Big(i
			Ve(s)-R\big(h^a_k(s)+e(s)\big)+R(h^a_k(s))+ \eps_k(s)\Big) ds.
			\endaligned
		\end{equation*}
		Thus the existence of a solution $U^a$ of \eqref{har} satisfying
		\eqref{difference:Stri} for $t\geq t_0$ is equivalent to the
		following fixed point problem
		\begin{equation}
			\label{defM} \aligned
			\forall \; t\geq t_0,\quad
			e(t)=\MMM_k(e)(t)\text{ and } \|  e \|_{l(t,+\infty)}  \leq
			e^{-\left(k+\frac{1}{2}\right)e_0 t}.
			\endaligned
		\end{equation}
		
		Let us fix $k$ and $t_0$. Denote
		\begin{align*}
			B_{l}^k&:=\big\{e\in Z(t_0,+\infty),\;\nabla e\in S(t_0,+\infty);\;
			\sup_{t\geq t_0} e^{\left(k+\frac{1}{2}\right)e_0 t}  \| e \|_{l(t,+\infty)} \leq 1\big\}.
		\end{align*}
		By using the Strichartz estimate and Lemmar \ref{linearoperator:prelimestimate}, we can show that if $t_0$ and $k$ are large enough, the
		mapping $\MMM_k$ is a contraction map on $B_l^k$.
		
		Finally, by the finite propagation of finite energy solutions, we can obtain that the analogue estimate $U^a\in L^2(\R^N)$ if $N\geq 5$ from the derivative estimate of the localized  mass to that in \cite{DuyMerle:NLS:ThresholdSolution}. For this purpose, define a positive radial function $\psi$ on $\R^N$,such that $\psi=1$ if $|x|\leq1$ and $\psi=0$ if $|x|\geq2$. We also define that for $R>0$ and large $t$,
		$$F_R(t):=\int_{\R^N}|U^a(t,x)|^2\psi(\frac{x}{R})dx.$$
		Then, we have
		\begin{align*}
			F_R^{'}(t)&=\frac{2}{R}Im\int U^a\nabla \bar{U}^a\cdot (\nabla \psi)(\frac{x}{R})dx
			=\frac{2}{R}Im\int W\nabla (\bar{U}^a-W)\cdot (\nabla \psi)(\frac{x}{R})dx\\
			&+\frac{2}{R}Im\int (U^a-W)\nabla W\cdot (\nabla \psi)(\frac{x}{R})dx
			+\frac{2}{R}Im\int (U^a-W)\nabla (\bar{U}^a-\bar{W})\cdot (\nabla \psi)(\frac{x}{R})dx
		\end{align*}
		Using(3.5), $\|U^a(t)-W\|_{\dot{H}^1}\leq Ce^{-e_0t}$ and the Hardy inequality, we get
		$$|F_R^{'}(t)|\leq C\|U^a(t)-W\|_{\dot{H}^1}(\|U^a(t)\|_{\dot{H}^1}+\|W\|_{\dot{H}^1})\leq Ce^{-e_0t},$$
		with a constant C independent od R, and thus, integrating between sufficiently large t and $+\infty$,
		$$\left|F_R(t)-\int_{\R^N}|W(x)|^2\psi(\frac{x}{R})dx\right|\leq Ce^{-e_0t}.$$
		Letting R goes to $+\infty$, we get $\|U^a(t)\|_{L^2}=\|W\|_{L^2}$ and $U^a(t)\in L^2(\R^N)$ when $N\geq5$, which completes the proof from mass conservation law.
	\end{proof}

	\subsection{Construction of	$W^{\pm}$.}

	Note that by \eqref{difference:Energy}, and the energy conservation,
	we have
	$$E(U^a)=E(W).$$
	In addition, by \eqref{difference:Energy}, we have
	\begin{equation*}
		\aligned \big\|\nabla U^a(t) \big\|^2_{2}= \big\|\nabla W \big\|^2_2
		+ 2a e^{-e_0 t}\int_{\R^N} \nabla W \cdot \nabla \mathcal{Y}_1 \;dx+
		O\left(e^{-\frac32e_0t}\right)\quad \text{as} \; t\rightarrow
		+\infty.
		\endaligned
	\end{equation*}
	By \eqref{WYrelation}, replacing $\mathcal{Y}_+$ by $-\mathcal{Y}_+$
	if necessary, we may assume that
	\begin{equation*}
		\aligned \int \nabla W \cdot \nabla \mathcal{Y}_1  \; dx > 0,
		\endaligned
	\end{equation*}
	which implies that $\big\|\nabla U^a(t) \big\|^2_{2} - \big\|\nabla
	W \big\|^2_2$ has the sign of $a$ for large positive time. Thus, by
	Lemma \ref{energytrapp}, $\big\|\nabla U^a (t_0) \big\|^2_{2} -
	\big\|\nabla W \big\|^2_2$ has the sign of $a$.

	Let
	\begin{equation}\label{Definition:threshold}
		\aligned   W^{+}(t,x)=  U^{+1}(t+t_0, x), \quad W^{-}(t,x)=
		U^{-1}(t+t_0, x),
		\endaligned
	\end{equation}
	which yields two radial solutions $W^{\pm}(t,x)$ of \eqref{har} for
	$t\geq 0$, and satisfies
	\begin{equation*}
		\aligned E(W^{\pm}(t))=E(W), \; \big\|\nabla W^-(0)\big\|_2 <
		\big\|\nabla W \big\|_2 , \; \big\|\nabla W^+(0)\big\|_2 >
		\big\|\nabla W \big\|_2,
		\endaligned
	\end{equation*}
	and
	\begin{equation*}
		\aligned  \big\|W^{\pm}(t)-W\big\|_{\dot H^1} \leq C e^{-e_0t}, \;\;
		t\geq 0.
		\endaligned
	\end{equation*}
	
	Moreover, we have $W^{\pm}\in L^2$ when $N\geq 5$.  To conclude the proof of Theorem \ref{threholdsolution}, it remains
	to show the dynamics of $W^{\pm}$ for the negative time, we will
	leave the proof in Section \ref{subs:blowup} and Section \ref{S:sub:scattering}. \qed

	%
	%
	%
	%
	
	\section{Modulation Analysis}\label{S:modulation}
	For the radial function $u\in \dot H^1(\R^N)$, we define the
	gradient variant from $W$ as
	\begin{equation*}
		\aligned \delta(u) = \left|\int_{\R^N} \Big( |\nabla u(x)|^2 -
		|\nabla W(x)|^2\Big) dx\right|.
		\endaligned
	\end{equation*}
	By Proposition \ref{P:static stability}, we know that if
	\begin{equation}\label{GSenergy}
		\aligned E(u)=E(W),
		\endaligned
	\end{equation}
	and $\delta(u)$ is small enough, then there exist
	$\widetilde{\theta}$ and $\widetilde{\mu}$ such that
	\begin{equation*}
		\aligned u_{\widetilde{\theta}, \widetilde{\mu}} = W +
		\widetilde{u}, \;\; \text{with}\;  \big\|\widetilde{u}\big\|_{\dot
			H^1} \leq \varepsilon(\delta(u)),
		\endaligned
	\end{equation*}
	where $u_{\widetilde{\theta}, \widetilde{\mu}}$ is defined as Proposition \ref{P:static stability}, and $\varepsilon(\delta) \rightarrow 0$ as $\delta\rightarrow 0$.
	The goal in this section is that for radial threshold solution of \eqref{har}, we
	choose parameters $\widetilde{\theta}$ and $\widetilde{\mu}$ to
	obtain the dynamics of these parameters and their
	derivatives in terms of the gradient variant $\delta(u)$.
	
	First, by the implicit theorem and the variational characterization
	of $W$ in Proposition \ref{P:static stability}, we have the
	following orthogonal decomposition, which is analogue to Lemma $4.1$ in \cite{MiaoWX:dynamic gHartree}.
	
	\begin{lemma}\label{choice:spatialtranslation}
		There exist  $\delta_0>0$ and a positive function $\epsilon(\delta)$
		defined for $0<\delta\leq \delta_0$, which tends to $0$ when $\delta
		$ tends to $0$, such that for all radial $u$ in $\dot H^1(\R^N)$
		satisfying \eqref{GSenergy}, there exists a couple $(\theta, \mu)\in
		\R\times (0, +\infty)$   such that $v=u_{\theta,\mu}$ satisfies
		\begin{equation}
			\aligned v\bot iW, \quad v \bot \widetilde{W}.
			\endaligned
		\end{equation}
		The parameters $(\theta, \mu)$ are unique in $\R/2\pi\Z \times \R,
		$ and the mapping $u\mapsto (\theta, \mu)$ is $C^1$.
	\end{lemma}
	\begin{proof} First by Proposition \ref{P:static stability}, there
		exist a function $\epsilon $ and $\theta_1, \mu_1>0$ such that
		$u_{\theta_1, \mu_1}=W+g$ and
		\begin{equation}\label{est:variational}
			\big\| u_{\theta_1, \mu_1} -W\big\|_{\dot H^1} \leq \epsilon \left(
			\delta(u)\right).
		\end{equation}
		
		Now consider the functional
		\begin{align*}
			J(\theta, \mu, f)=&  \left(J_0(\theta, \mu, f), J_1(\theta, \mu, f)
			\right) = \left( (f_{\theta, \mu}, iW)_{\dot H^1}, (f_{\theta,\mu},
			\widetilde{W})_{\dot H^1}\right)\\
			= & \left( (e^{i\theta}\mu^{-\frac{N-2}{2}} f(x/\mu), iW)_{\dot
				H^1}, (e^{i\theta}\mu^{-\frac{N-2}{2}} f(x/\mu),
			\widetilde{W})_{\dot H^1}\right).
		\end{align*}
		Hence, by simple computations, we have
		\begin{align*}
			J(0,1, W)=0, \quad \left| \frac{\partial J}{\partial(\theta, \mu)}
			(0,1, W)\right| = \left| \begin{array}{cc} \big\|\nabla W \big\|^2_{L^2} & 0\\
				0 &  - \big\|\nabla \widetilde{W} \big\|^2_{L^2} \end{array} \right|
			\not = 0.
		\end{align*}
		Thus by the implicit theorem, there exist $\epsilon_0, \eta_0 > 0$,
		we have that for any $h\in \dot H^1(\R^N)$ with $\big\| h- W \big\|_{\dot
			H^1} < \epsilon_0$, there exists a unique
		$\left(\widetilde{\theta}(h), \widetilde{\mu}(h)\right) \in C^1$
		such that $|\widetilde{\theta}| + |\widetilde{\mu}-1 |< \eta_0 $ and
		\begin{align*}
			J(\widetilde{\theta}, \widetilde{\mu}, h)=0.
		\end{align*}
		By \eqref{est:variational}, this implies that there exists a unique
		$\left(\widetilde{\theta}_1(u), \widetilde{\mu}_1(u)\right)$ such
		that
		\begin{align*}
			J\left(\widetilde{\theta}_1 + \theta_1, \widetilde{\mu}_1 \mu_1, u
			\right)= J\left(\widetilde{\theta}_1, \widetilde{\mu}_1,
			u_{\theta_1, \mu_1} \right) = 0.
		\end{align*}
		This completes the proof by taking $\theta =\widetilde{\theta}_1 +
		\theta_1$ and $\mu = \widetilde{\mu}_1 \mu_1$.
	\end{proof}
	
	Let $u(t)$ be a radial solution of \eqref{har} satisfying
	\eqref{GSenergy} and write
	\begin{equation}\label{def:delta}
		\aligned \delta(t) := \delta(u(t))= \left|\int_{\R^N}\Big( |\nabla
		u(t, x)|^2 - |\nabla W(x)|^2 \Big) dx\right|.
		\endaligned
	\end{equation}
	
	Let $D_{\delta_0}$ be the open set of all time $t$ in the domain of
	existence of $u$ such that  $\delta(t)< \delta_0$. On
	$D_{\delta_0}$, by Lemma \ref{choice:spatialtranslation}, there
	exist $\theta(t)$ and $\mu(t)$, which are $C^1$ functions of $t$
	such that we have the following decomposition
	\begin{equation}\label{modulcomp}
		\aligned   u_{\theta(t), \mu(t)}(t, & x ) = \Big( 1+ \alpha(t) \Big)
		W(x) + h(t,x),
		\endaligned
	\end{equation}
	\begin{equation*}
		\aligned  1+ \alpha(t) =& \frac{1}{\big\|W\big\|^2_{\dot
				H^1}}\left(u_{\theta(t), \mu(t)}, W\right)_{\dot H^1},\quad h \in
		H^{\bot}\cap \dot H^1_{\text{rad}}.
		\endaligned
	\end{equation*}
	We will make use of additional conditions
	to show that $u$ converges exponentially  to $W$ in the positive time in $\dot H^1(\R^N)$ up to
	the constant modulation parameters in Section \ref{S:convergence:sup} and Section
	\ref{S:convergence:sub}.  The reason why we make this decomposition is that the linearized energy $\Phi$ is coercive on $H^{\bot}\cap \dot H^1_{\text{rad}}$.

	\begin{lemma}\label{estmodulpara}Let $u$ be a radial
		solution of \eqref{har}
		satisfying \eqref{GSenergy}. Then taking a smaller $\delta_0$ if
		necessary, the following estimates hold for $t \in D_{\delta_0}$:
		\begin{align}\label{modulpara}
			\big| \alpha(t)\big| \approx \big\|   \alpha(t)W(\cdot) +
			h(t,\cdot)\big\|_{\dot H^1}\approx & \big\|h(t,\cdot)\big\|_{\dot
				H^1} \approx
			\delta(t) ,\\
			\label{modulparaderiva}  \big| \alpha'(t)\big| + \big|
			\theta'(t)\big| +\left|\frac{\mu'(t)}{\mu(t)}\right|\leq &\; C
			\mu^2(t)\delta(t).
		\end{align}
	\end{lemma}
	\begin{proof}[Sketch of the proof]
		The proof mainly uses the coercivity of the linearized energy $\Phi$ on $H^{\bot}\cap \dot H^1_{rad}$ in Proposition \ref{coerH} and the orthogonal decomposition in Lemma \ref{choice:spatialtranslation}, we can refer to Lemma 4.2 in \cite{MiaoWX:dynamic gHartree} for more details.
	\end{proof}

	%
	%
	%
	%
	
	\section{Convergence to $W$ for the supercritical threshold solution}\label{S:convergence:sup}
	
	In this section, we follows the argument in Section $5$ in \cite{MiaoWX:dynamic gHartree} to show the following result, which implies the
	dynamics of special threshold solution $W^+$ of Theorem \ref{threholdsolution} in the negative
	time and is also the first step in the proof of case (c) of Theorem
	\ref{classification}.
	
	\begin{proposition}\label{expdecay:supercase}
		Consider a radial solution $u \in H^1(\R^N)$ of \eqref{har} such
		that
		\begin{equation}\label{data:super}
			\aligned E(u)=E(W), \quad \big\| \nabla u_0\big\|_{L^2}>
			\big\|\nabla W\big\|_{L^2},
			\endaligned
		\end{equation}
		which is globally defined for the positive time. Then there exist
		$\theta_0\in \R/(2\pi\Z), \mu_0 \in (0, +\infty), c, C>0$ such that
		\begin{equation}\label{asymptotic:supercritical case}
			\aligned \forall \; t\geq 0, \quad \big\|u-W_{\theta_0,
				\mu_0}\big\|_{\dot H^1} \leq Ce^{-ct},
			\endaligned
		\end{equation}
		and the negative time of existence of $u$ is finite.
	\end{proposition}
	
	\subsection{Exponential convergence of the gradient variant}
	
	\begin{lemma}[\cite{MiaoWX:dynamic gHartree}\label{L:firder}] Under the assumptions of Proposition
		\ref{expdecay:supercase}, then there exists $C$ such that for any
		$R>0$, and all $t\geq 0$,
		\begin{align}
			\label{Fderiv}   \big| \partial_t V_R(t) \big| \leq   CR^2 \;
			\delta(t),
		\end{align}
		where $V_R(t)$ is defined as in \eqref{localV}.
	\end{lemma}
	\begin{proof}By Lemma \ref{localV}, Lemma \ref{energytrapp} and Lemma \ref{estmodulpara},  we can obtain the result for both large $\delta(t)$ and small $\delta(t)$, which is similar to Lemma $5.2$ in \cite{MiaoWX:dynamic gHartree}. We omit the details here.
	\end{proof}
	
	
	\begin{lemma}\label{L:secderiv} Under the assumptions of Proposition
		\ref{expdecay:supercase},  there exist $C>0$ and $R_1\geq 1$
		such that for $R\geq R_1$, and all $t\geq 0$,
		\begin{align}
			\label{Secderiv} \partial^2_t V_R(t) \leq & - 2(p-1) \delta(t),\\
			\label{fderiv:positive}
			\partial_t V_R(t) >&  \; 0.
		\end{align}
	\end{lemma}
	\begin{proof}
		By Lemma \ref{L:local virial}, we have
		\begin{align*}
			\partial^2_t V_R(t)=&\; 4 \int_{\R^N} \Big|\nabla u (t,x)\Big|^2 dx -4
			\int_{\R^N}  \left(\ I_{\lambda}*|u|^p\right)\; |u(x)|^p \; dx
			+ A_R\big(u(t)\big),
		\end{align*}
		where $ A_R\big(u(t)\big)$ is defined in Lemma \ref{L:local virial}.
		
		By $E(u)=E(W) = \frac{p-1}{2p}\big\|\nabla W \big\|^2_{L^2}$, we have
		\begin{align*}
			4 \int_{\R^N} \Big|\nabla u (t,x)\Big|^2 dx -4
			\int_{\R^N}  \left(\ I_{\lambda}*|u|^p\right)\; |u(x)|^p \; dx=-4(p-1) \delta(t).
		\end{align*}
		To prove \eqref{Secderiv}, it suffices to show that for $R \geq R_1$
		\begin{align*}
			A_R\big(u(t)\big)  \leq 2(p-1) \delta(t).
		\end{align*}
		We divide the proof into three steps:
		
		\noindent{\bf Step 1: General bound on $A_R(u(t))$.} We claim that
		there exists $C>0$ such that for any  $R>0$, $t\geq 0$, we have
		\begin{align}\label{errcontrol1:sup}
			A_R\big(u(t)\big) \leq \frac{C}{R^2} + \frac{C}{R^{\frac{p(N-1)}{N}}}
			\big\|u(t)\big\|^{\frac{p(N+1)}{N}}_{\dot H^1}.
		\end{align}
		Indeed, on one hand,  by the facts that
		$
		\phi''(r) \leq 1,\; r\geq 0$, $ \big| -\Delta \Delta
		\phi_R \big| \lesssim  R^{-2},
		$
		we have
		\begin{align}\label{est:Ar1}
			A_R\big(u(t)\big) \leq C \left(\frac{\big\|u(t)\big\|^2_{L^2}}{R^2}
			+   \int_{|x|>R} (I_{\lambda}*|u|^p)|u|^{p} \; dx\right).
		\end{align}
		On the other hand,	by the Hardy-Littlewood-Sobolev inequality, mass conservation and the Sobolev inequality, we have
		\begin{align*}
			\int_{|x|>R} (I_{\lambda}*|u|^p)|u|^{p} \; dx &\lesssim \big\|u(t)\big\|^p_{L^{\frac{2N}{N-2}}(|x|>R)}\big\|u(t)\big\|^p_{L^{\frac{2N}{N-2}}}\\
			&\lesssim \big\|u(t)\big\|^p_{L^{\frac{2N}{N-2}}}\big\|u(t)\big\|^{\frac{2p}{N}}_{L^{\infty}(|x|>R)}\big\|u(t)\big\|^{\frac{pN-2p}{N}}_{L^2}\\
			&\lesssim \big\|\nabla u(t)\big\|^p_{L^2}\big\|u(t)\big\|^{\frac{2p}{N}}_{L^{\infty}(|x|>R)}\\
			&\lesssim \big\|\nabla u(t)\big\|^p_{L^2}\left(\frac{1}{R^{N-1}}\big\|\nabla u(t)\big\|_{L^2}\big\| u(t)\big\|_{L^2}\right)^{\frac{p}{N}}\\
			&\lesssim \frac{1}{R^{\frac{p(N-1)}{N}}}\big\|\nabla u(t)\big\|_{L^2}^{\frac{p(N+1)}{N}}.
		\end{align*}
		which together with \eqref{est:Ar1} gives \eqref{errcontrol1:sup}.
		
		\noindent{\bf Step 2: Refined bound on $A_R(u(t))$ when $\delta(t)$
			small.} We claim that there exist $\delta_1$,  $C>0$ such that for
		any $R > 1$, and $t\geq 0$, and $\delta(t)\leq \delta_1$,  we have
		\begin{align}\label{errcontrol2:sup}
			\Big| A_R(u(t)) \Big| \leq C \left( \frac{1}{ R ^{\frac{N-2}{2}}}
			\delta(t) + \delta(t)^2\right).
		\end{align}
		
		To do so, we first show that for small $\delta_1$,
		\begin{align}\label{sf:sup}
			\mu_-:=\inf \Big\{\mu(t), t\geq 0, \delta(t)\leq \delta_1 \Big\}
			> 0.
		\end{align}
		Indeed, by \eqref{modulcomp} and Lemma \ref{estmodulpara}, we have
		\begin{align*} u_{\theta(t), \mu(t)} (t)= W + V,\; \text{with}\;\;
			\big\|V(t)\big\|_{\dot H^1} \approx \delta(t).
		\end{align*}
		By the mass conservation, we have
		\begin{align*}
			\big\|u_0\big\|^2_{L^2} \geq & \int_{|x|\leq \frac{1}{\mu(t)}} \big| u(t,x)
			\big|^2 dx = \frac{1}{\mu(t)^2} \int_{|x|\leq 1} \big| u_{\theta(t),
				\mu(t)}\big|^2 dx \\
			\gtrsim & \frac{1}{\mu(t)^2} \left(\int_{|x|\leq 1}W^2 dx
			- \int_{|x|\leq 1}\big|V(t) \big|^2 dx \right).
		\end{align*}
		
		While by the Sobolev inequality, we have
		\begin{align*}
			\big\|V(t)\big\|_{L^2(|x|\leq 1)} \lesssim
			\big\|V(t)\big\|_{L^{\frac{2N}{N-2}}(|x|\leq 1)} \lesssim
			\big\|V(t)\big\|_{\dot H^1} \lesssim \delta(t),
		\end{align*}
		which implies that
		\begin{align*}
			\big\|u_0\big\|^2_{L^2} \geq  \frac{1}{\mu(t)^2} \left(
			\int_{|x|\leq 1} W^2 dx -  C \delta(t)^2 \right).
		\end{align*}
		We obtain \eqref{sf:sup} if choosing sufficiently small $\delta_1$.
		Now by  \eqref{modulcomp} and the change of variable, we have
		\begin{align*}
			\Big| A_R(u(t)) \Big| = \Big| A_R\big((W+V)_{-\theta(t),1/\mu(t) }
			\big) \Big| = \Big| A_{R\mu(t)}\big( W+V   \big) \Big|.
		\end{align*}
		Note that
		\begin{align*}
			A_{R\mu(t)}(W)=0,\;\text{and}\;\; \big\|\nabla W\big\|_{L^2(|x|\geq
				\rho)}\approx \big\| W\big\|_{L^{\frac{2N}{N-2}}(|x|\geq
				\rho)}\approx \rho^{-\frac{N-2}{2}}\; \text{for}\; \rho\geq 1,
		\end{align*}
		we have
		\begin{align*}
			\Big| A_R(u(t)) \Big|   = & \Big| A_{R\mu(t)}(W+V)-
			A_{R\mu(t)}(W) \Big|
			\\
			\lesssim &\;  \int_{|x|\geq R\mu(t)}\big| \nabla W \cdot \nabla V(t)
			\big| + \big| \nabla V(t)\big|^2
			dx \\
			& + \int_{R\mu(t)\leq |x|\leq 2R\mu(t)}
			\frac{1}{\big(R\mu(t)\big)^2} \left(\big|
			W \cdot V(t) \big| + \big| V(t)\big|^2  \right) dx \\
			& + \iint_{|x|\geq R\mu(t) }
			I_{\lambda}(x-y)\big|W(x)+V(t,x)\big|^{p}\big|W(y)+V(t,y)\big|^p\\
			&\qquad \quad  \qquad  \quad - I_{\lambda}(x-y)W(x)^{p}W(y)^p\; dxdy\\
			\lesssim & \frac{1}{\big(R\mu(t)\big)^{\frac{N-2}{2}}} \big\|V(t)\big\|_{\dot
				H^1} + \big\|V(t)\big\|^2_{\dot
				H^1} + \big\|V(t)\big\|^3_{\dot
				H^1} + \big\|V(t)\big\|^{2p}_{\dot
				H^1},
		\end{align*}
		which together with \eqref{sf:sup} implies that
		\eqref{errcontrol2:sup} if $\delta_1$ is small enough.

		\noindent{\bf Step 3: Conclusion.} From \eqref{errcontrol2:sup},
		there exist $\delta_2 > 0 $ and $R_2\geq 1$ such that if $R\geq R_2$, $\delta(t)\leq
		\delta_2$, then we have
		\begin{align*}
			\Big| A_R(u(t)) \Big| \leq C \left( R ^{-\frac{N-2}{2}} \delta(t) +
			\delta(t)^2\right) \leq 2(p-1) \delta(t).
		\end{align*}
		
		Now Let
		\begin{align*}
			f_{R_3}(\delta):=\frac{C}{R_3^2} + \frac{C}{R_{3}^{\frac{p(N-1)}{N}}}
			\big( \delta+ \big\|W\big\|^2_{\dot H^1}\big)^{\frac{p(N+1)}{2N}}-2(p-1)\delta,
		\end{align*}
		where $C$ is given by \eqref{errcontrol1:sup}. Choosing $R_3 $
		large enough such that $f_{R_3}(\delta_3)\leq 0$, and
		$f'_{R_3}(\delta_3)\leq 0$, we have for any $R\geq R_3$, $\delta\geq
		\delta_3$ that
		\begin{align*}
			f_{R}(\delta)\leq f_{R_3}(\delta) \leq 0,
		\end{align*}
		which implies that $ A_R(u(t)) \leq 2(p-1) \delta(t)$, so we conclude the
		proof of \eqref{Secderiv} with $R_1=\max(R_2, R_3)$.
		
		Finally, using the positivity of $V_R(t)$, $\partial^2_t
		V_R(t) < 0$ for $t>0$ and $u$ is defined on $[0, +\infty)$, we have $\partial_t V_R(t) > 0$  by the convexity analysis.
	\end{proof}
	
	\begin{lemma}\label{L:inteexpondecay} Under the assumptions of Proposition
		\ref{expdecay:supercase},  there exist $c>0$, $C>0$   such that
		for $R\geq R_1$ ( which is given in Lemma \ref{L:secderiv}), and all
		$t\geq 0$,
		\begin{align}
			\label{expdecay:Inte:supercase}
			\int^{+\infty}_t  \delta(s)ds   \leq & \; C  e^{-ct}.
		\end{align}
	\end{lemma}
	
	\begin{proof}
		Fix $R\geq R_1$. By \eqref{Fderiv}, \eqref{Secderiv} and
		\eqref{fderiv:positive}, we have 
		\begin{align*} 
			2(p-1)\int^T_t
			\delta(s)ds \leq - \int^T_t
			\partial^2_s V_R(s)ds =
			\partial_tV_R(t) - \partial_tV_R(T) \leq CR^2 \delta(t).
		\end{align*}
		Let $T\rightarrow +\infty $, we have
		\begin{align*}
			\int^{+\infty}_t  \delta(s)ds   \leq C  \delta(t).
		\end{align*}
		By the Gronwall inequality, we have \eqref{expdecay:Inte:supercase}.
	\end{proof}

	\subsection{Convergence of the modulation parameters.} By the variational characterization of W in Proposition \ref{P:static stability}, the relation between modulation parameters and the gradient variant which is shown in Lemma \ref{estmodulpara}, and Lemma \ref{L:inteexpondecay}, we can follow the argument as that  in Section $5.2$ in \cite{MiaoWX:dynamic gHartree} to obtain the convergence of the modulation parameters. Thus, we have
	\begin{align}
		\lim_{t\rightarrow+\infty}\delta(t)=0.\;\; &
		\lim_{t\rightarrow+\infty}\theta(t) =\theta_{0},\;\; \lim_{t\rightarrow+\infty}\mu(t) =\mu_{0} \label{diff:Gradient}\\
		\delta(t) \approx \big| \alpha(t) \big| = & \big| \alpha(t)
		-\alpha(+\infty) \big| \notag\\
		\leq & C \int^{+\infty}_t \big| \alpha'(s)\big| ds \leq C
		\int^{+\infty}_t  \mu^2(s) \delta(s) ds \leq C
		e^{-ct},
	\end{align}
	\begin{equation*}
		\aligned \big\|u-W_{\theta(t), \mu(t)}\big\|_{\dot H^1} + \big|
		\alpha'(t)\big| + \big| \theta'(t)\big| \leq C \delta(t) \leq C
		e^{-ct},
		\endaligned
	\end{equation*}
	which yields \eqref{asymptotic:supercritical case}. \qed

	\subsection{Blowup for the negative times}\label{subs:blowup}
	
	It is a consequence of the positivity of $\partial_t V_R(t)$ in
	\eqref{Fderiv} and the time reversal symmetry. 
	Suppose that $u$ is also global for the negative time. Applying
	Lemma \ref{L:firder}, Lemma \ref{L:secderiv} and Lemma
	\ref{L:inteexpondecay} to $\overline{u}(-t)$, we know that they also
	hold for the negative time.  Hence by \eqref{diff:Gradient}, we
	know that
	\begin{equation*}
		\lim_{t\rightarrow \pm \infty} \delta(t) =0.
	\end{equation*}
	By Lemma \ref{L:firder} and Lemma \ref{L:secderiv}, we know that
	$\partial_t V_R(t)> 0$ and $
	\partial_t V_R(t) \longrightarrow 0, \; \text{as}\; t\rightarrow \pm
	\infty.
	$
	By Lemma \ref{L:secderiv}, we have
	$
	\partial^2_t V_R(t)<0.
	$ This implies that $\partial_tV_R(t)\equiv 0$. It is a
	contradiction and completes the proof. \qed

	
	\section{Convergence to $W$ for the subcritical threshold solution}\label{S:convergence:sub}
	In this section, we consider the subcritical  threshold
	solution $u$ of \eqref{har}. Similar to that in Section
	\ref{S:convergence:sup}, the following result will give the dynamics
	of special threshold solution $W^-$ of Theorem \ref{threholdsolution} in the negative time and
	is also the first step in the proof of case (a) of Theorem
	\ref{classification}.
	\begin{proposition}\label{expdecay:subcase}
		Let $u \in \dot H^1(\R^N)$ be a radial solution of \eqref{har}, and
		$I=(T_-, T_+)$ be its maximal interval of existence. Assume that
		\begin{equation} \label{sub thresh case}
			\aligned E(u_0)=E(W),  \quad \big\| \nabla u_0\big\|_{L^2} <
			\big\|\nabla W\big\|_{L^2}.
			\endaligned
		\end{equation}
		Then $$I=\R.$$ Furthermore, if $u$ does not scatter for the positive
		time, that is,
		\begin{equation}\label{Noscattering}
			\aligned \big\|u\big\|_{Z(0,+\infty)}=\infty,
			\endaligned
		\end{equation}
		then
		there exist $\theta_0\in \R, \mu_0>0, c, C>0$ such that
		\begin{equation*}\label{asymptotic:subcritical case}
			\aligned \forall \;t\geq 0, \quad \big\|u-W_{\theta_0,
				\mu_0}\big\|_{\dot H^1} \leq Ce^{-ct},
			\endaligned
		\end{equation*}and
		\begin{equation*}\label{negscattering}
			\aligned \big\|u\big\|_{Z(-\infty, 0)}< \infty.
			\endaligned
		\end{equation*}
		An analogous assertion holds on $(-\infty, 0]$.
	\end{proposition}
	
	\subsection{Compactness properties.}\label{subs:compact} We recall some key results of 
	Theorem \ref{belowthreshold} and Remark \ref{belowthreshold:blow up}. Note
	that the fact that the solution with
	\begin{equation*}
		\aligned E(u)<E(W),  \quad \big\| \nabla u_0\big\|_{L^2} <
		\big\|\nabla W\big\|_{L^2},
		\endaligned
	\end{equation*}
	is global well-posed and scatters in both time directions. By Lemma
	\ref{energytrapp}, the concentration compactness principle and the stability theory as the proof of Proposition $4.2$ in \cite{MiaoXZ:09:e-critical radial Har} and Lemma $6.2$ in \cite{MiaoWX:dynamic gHartree}, we can show the following pre-compact property of minimal energy blow-up solution.
	
	\begin{lemma}[\cite{MiaoWX:dynamic gHartree, MiaoXZ:09:e-critical radial Har}] \label{compactness}
		Let $u$ be a radial solution of \eqref{har} satisfying
		\begin{align*}
			E(u_0)=E(W),  \;\; \big\| \nabla u_0\big\|_{L^2} < \big\|\nabla
			W\big\|_{L^2}, \; \; \big\|u\big\|_{Z(0,T_+)}=+\infty.
		\end{align*}   Then there
		exists a continuous functions $\lambda(t)$ such that the set
		\begin{equation}\label{Kcompact}
			\aligned K:= \big\{ (u(t))_{\lambda(t)} , t\in [0,T_+) \big\}
			\endaligned
		\end{equation}
		is  pre-compact in $\dot H^1_{\text{rad}}(\R^N)$, where $(u(t))_{\lambda(t)}:=\lambda(t)^{-\frac{N-2}{2}}u(t,\frac{x}{\lambda(t)})$.
	\end{lemma}
	
	Let $u$ be a solution of \eqref{har}, and $\lambda(t)$ be as in
	Lemma \ref{compactness}. Consider $\delta_0$ as in Section 4. The
	parameters $\theta(t)$, $\mu(t)$ and $\alpha(t)$ are defined for
	$t\in D_{\delta_0}= \{t: \delta(t)< \delta_0 \}$. By
	\eqref{modulcomp} and Lemma \ref{estmodulpara}, there exists a
	constant $C_0>0$ such that for any $ t \in D_{\delta_0},$, we have
	\begin{equation*}
		\int_{\frac{1}{\mu(t)}\leq |x|\leq
			\frac{2}{\mu(t)}}\big|\nabla u(t,x) \big|^2 dx =  \int_{1\leq |x|\leq
			2} \frac{1}{\mu(t)^N} \left| e^{i\theta(t)}\nabla u\left(t,\frac{x}{\mu(t)}\right) \right|^2 dx  
		\geq \int_{1\leq |x|\leq 2}\big|
		\nabla W\big|^2 -C_0 \delta(t).
	\end{equation*}
	Taking a smaller $\delta_0$ if necessary, we can assume that the
	right hand side of the above inequality is bounded from below by a
	positive constant $\varepsilon_0$ on $D_{\delta_0}$. Thus, we have
	\begin{equation*}
		\int_{\frac{\lambda(t)}{\mu(t)}\leq |x|\leq \frac{2
				\lambda(t)}{\mu(t)}}\frac{1}{\lambda(t)^N}\left|\nabla
		u\left(t,\frac{x}{\lambda(t)}\right) \right|^2 dx \geq \int_{1\leq
			|x|\leq 2}\big| \nabla W\big|^2 -C_0 \delta(t), \;\text{for any}\; t \in D_{\delta_0}.
	\end{equation*}
	By the compactness of $\overline{K}$, it follows that for any $t\in
	D_{\delta_0}$, we have $\big| \mu(t) \big| \sim \big| \lambda(t)
	\big| $. As a consequence, we may modify $\lambda(t)$ such that $K$
	defined by \eqref{Kcompact} remains pre-compact in $\dot H^1$ and for any $ t\in D_{\delta_0}$, we have
	\begin{equation}\label{para-uniform-SGD}
		\lambda(t)=\mu(t).
	\end{equation}
	
	As a consequence of Lemma \ref{compactness}, we have
	\begin{corollary}\label{globalpositive}Let $u$ be a radial solution of \eqref{har} satisfying \eqref{sub thresh
			case} and  not scatter for the positive time. Then
		\begin{enumerate}
			\item[\rm (a)] $T_+=+\infty.$
			
			\item[\rm (b)] $\displaystyle \lim_{t \rightarrow \infty} \sqrt{t} \lambda(t) =\infty$;
		\end{enumerate}
	\end{corollary}
	
	As a direct consequence of (a) in Corollary \ref{globalpositive}, we
	have
	
	\begin{corollary}\label{global}
		Let $u$ be a radial solution of \eqref{har} with the maximal lifespan
		interval $I$ and  satisfy
		\begin{equation*} \aligned E(u_0)\leq E(W), \quad \big\|\nabla
			u_0\big\|_2 \leq \big\|\nabla W \big\|_{2},
			\endaligned
		\end{equation*}
		then $I=\R.$
	\end{corollary}
	The proofs of these results are same as the proofs of Corollary $6.3$ and Corollary $6.4$ in \cite{MiaoWX:dynamic gHartree}, which is the consequence of the compactness of $\overline{K}$ in $\dot H^1(\R^N)$, so we omit the details.

	\subsection{Convergence in the ergodic mean.}
	\begin{lemma}\label{meanconverg}
		Let $u$ be a radial solution of \eqref{har} satisfying  \eqref{sub
			thresh case} and  \eqref{Noscattering}. Then
		\begin{equation*}
			\aligned \lim_{T\rightarrow +\infty}
			\frac{1}{T}\int^{T}_{0}\delta(t) dt =0,
			\endaligned
		\end{equation*}
		where $\delta(t)$ is defined by \eqref{def:delta}.
	\end{lemma}
	\begin{proof} By Lemma \ref{L:local virial} and the Hardy inequality, we have
		\begin{align*}
			\big| \partial_t V_R(t)\big| \leq & C R^2.
		\end{align*}
		In addition, by $E(u)=E(W)=\frac{p-1}{2p}\big\|\nabla W\big\|^2_{L^2}$, we have
		\begin{align*}
			\partial^2_t V_R(t)=&  4(p-1) \delta(t) + A_R(t),
		\end{align*}
		where $ A_R\big(u(t)\big)$ is defined in Lemma \ref{L:local virial}.
		By the compactness of $\overline{K}$ in $\dot H^1$, for any
		$\epsilon >0$, there exists $\rho_{\epsilon} >0$ such that
		\begin{align*}
			\int_{|x|\geq \frac{\rho_{\epsilon}}{\lambda(t)}} \big| \nabla
			u(t,x) \big|^2 dx \leq \epsilon,
		\end{align*}
		which together with Lemma \ref{L:secderiv} implies  for $\displaystyle R \geq  \rho_{\epsilon}/\lambda(t)$, that 
		\begin{align}\label{smalremainder}
			\forall \; t\geq 0,\;  \big| A_R(t) \big| \leq \epsilon.
		\end{align}
		
		Fix $\epsilon$, choosing $\epsilon_0$ and $M_0$ such that
		\begin{align*}
			2C \epsilon^2_0 = \epsilon,\;\; M_0\epsilon_0 \geq \rho_{\epsilon}.
		\end{align*}
		By Corollary \ref{globalpositive} (b), there exists $t_0\geq 0$ such
		that
		\begin{align*}
			\forall \; t\geq t_0, \;\; \lambda(t) \geq \frac{M_0}{\sqrt{t}}.
		\end{align*}
		For $T\geq t_0$, let
		\begin{align*}
			R:=\epsilon_0 \sqrt{T}.
		\end{align*}
		For $t\in [t_0, T]$, we have
		\begin{align*}
			R = \epsilon_0 \sqrt{T} \; \frac{ M_0 }{\sqrt{t} \lambda(t)} =
			\frac{\sqrt{T}}{\sqrt{t}}\; \frac{M_0 \epsilon_0}{ \lambda(t)} \geq
			\frac{\rho_{\epsilon}}{\lambda(t)},
		\end{align*}
		which yields that
		\begin{align*}
			4(p-1)\int^T_t \delta(s) ds \leq & \int^T_t \partial^2_s V_R(s) ds +
			\big| A_R(s) \big| \big( T -t_0\big) \leq   2CR^2 + \epsilon T = 2
			\epsilon T.
		\end{align*}
		Let $T\rightarrow +\infty$, we have
		\begin{align*}
			\lim_{T\rightarrow+\infty} \frac{1}{T} \int^T_0 \delta(s) ds \leq
			\frac{\epsilon}{2},
		\end{align*}
		which concludes the proof.
	\end{proof}

	\begin{corollary}\label{seqconv}
		Let $u$ be a radial solution of \eqref{har} satisfying  \eqref{sub
			thresh case} and  \eqref{Noscattering}.  Then there exists a
		sequence $t_n$ such that $t_n \rightarrow +\infty$ and
		\begin{equation*}
			\aligned \lim_{n\rightarrow +\infty} \delta(t_n)=0.
			\endaligned
		\end{equation*}
	\end{corollary}

	\subsection{Exponential convergence.}\label{subs:expconv}
	
	As the same with Section \ref{S:convergence:sup}, we will give the exponential convergence of the gradient variant, which will prove the Proposition \ref{expdecay:subcase}. The two ingredients of the proof of Proposition \ref{expdecay:subcase} are the local virial argument (lemma \ref{virialestimate-delta:positive}) and a precise control of the variations of the parameter $\lambda(t)$ (Lemma \ref{control:parameter}).
	\begin{lemma}\label{virialestimate-delta:positive} Let $u$ be a
		radial solution of \eqref{har} satisfying  \eqref{sub thresh case},
		\eqref{Noscattering} and \eqref{para-uniform-SGD}, and $\lambda(t)$
		be as in Lemma \ref{compactness}. Then there exists a constant $C$
		such that if $\; 0\leq \sigma < \tau$, we have
		\begin{equation*}
			\aligned \int^{\tau}_{\sigma} \delta (t) dt \leq C \left(
			\sup_{\sigma\leq t\leq \tau}\frac{1}{\lambda(t)^2} \right) \Big(
			\delta(\sigma) + \delta(\tau) \Big).
			\endaligned
		\end{equation*}
	\end{lemma}
	The proof is analogue to that of Lemma 6.7 in \cite{MiaoWX:dynamic gHartree}. For the sake of completeness, We give the proof.
	\begin{proof} 
		For $R>0$, Let us consider the function $V_R(t)$ defined as in
		\eqref{localV}.
		
		By the same estimate as that in Lemma \ref{L:firder}, there is a
		constant $C_0$ independent of $t\geq 0$ such that
		\begin{align*}
			\Big| \partial_t V_R(t) \Big| \leq C_0 R^2 \delta(t).
		\end{align*}
		
		Now we show that if $\epsilon>0$, there exists $R_{\epsilon}$ such
		that for any $R \geq R_{\epsilon}/\lambda(t)$, then
		\begin{align}\label{Sderiv:subcase}
			\partial^2_t V_R(t) \geq  (4(p-1)-\epsilon) \delta(t).
		\end{align}
		Indeed, by Lemma \ref{L:local virial} and $E(u)=E(W) = \frac{p-1}{2p}\big\|\nabla
		W \big\|^2_{L^2}$, we have
		\begin{align*}
			\partial^2_t V_R(t)=&  4(p-1) \delta(t) + A_R(t),
		\end{align*}
		where $A_R\big(u(t)\big)$ is defined in Lemma \ref{L:local virial}.
		
		To prove \eqref{Sderiv:subcase}, it suffices to show that if
		$\epsilon>0$, there exists $R_{\epsilon}$ such that for any $R \geq
		R_{\epsilon}/\lambda(t)$,
		\begin{align*}
			\Big| A_R(t) \Big|\lesssim \epsilon
			\delta(t).
		\end{align*}
		
		For the case when $\delta(t)\geq \delta_0$. As the estimate
		\eqref{smalremainder}, we can use the compactness of $\overline{K}$
		in $\dot H^1$ to show that for any $t\geq 0$, $\epsilon>0$, there
		exists $\rho_{\epsilon}>0$, such that for any $R\geq
		\rho_{\epsilon}/\lambda(t)$, we have
		\begin{align*}
			\big|A_R(t) \big|\leq \epsilon\lesssim \epsilon \delta(t).
		\end{align*}
		
		For the case when $\delta(t)< \delta_0$. As the proof of Step 2 in
		Lemma \ref{L:secderiv}, we can show that for any $t\geq 0$,
		$\rho>1$, there exists $C>0$ such that for any $R\geq
		\rho/\lambda(t)$,  we have
		\begin{align*}
			\big|A_R(t) \big|\leq C\left( \rho^{-\frac{N-2}{2}} \delta(t)+
			\delta(t)^2\right),
		\end{align*}
		which implies \eqref{Sderiv:subcase} if choosing $\rho$ large
		enough.
		
		By \eqref{Sderiv:subcase}, there exists $R_2$ such that for any
		$R\geq R_2/\lambda(t)$,
		\begin{align*}
			\partial^2_t V_R(t) \geq 2(p-1) \delta(t).
		\end{align*}
		
		Finally, if taking $R = R_2 \displaystyle \sup_{\sigma\leq t \leq
			\tau} \left(\frac{1}{\lambda(t)}\right),$ and integrating between
		$\sigma$ and $\tau$, we have
		\begin{align*}
			2(p-1)\int^{\tau}_{\sigma} \delta(t) dt \leq \int^{\tau}_{\sigma}
			\partial^2_t V_R(t) dt = \partial_t V_R(\tau)-\partial_t V_R(\sigma)
			\leq C R^2 \Big(\delta(\tau) + \delta(\sigma) \Big),
		\end{align*}
		which implies the result. 
	\end{proof}
	
	In order to get the exponential convergence of the gradient variant, we need to control the variations of $\frac{1}{\lambda(t)^2}$.
	
	\begin{lemma}[\cite{MiaoWX:dynamic gHartree}]\label{control:parameter}
		Let $u$ be a radial solution of \eqref{har} satisfying \eqref{sub
			thresh case}, \eqref{Noscattering} and \eqref{para-uniform-SGD},
		and $\lambda(t)$ be as in Lemma \ref{compactness}. Then there is a
		constant $C_0>0$ such that for any $ \sigma, \tau>0$ with
		$\sigma+\frac{1}{\lambda(\sigma)^2} \leq \tau$, we have
		\begin{equation*}\label{parcon}
			\aligned \left|\frac{1}{\lambda(\tau)^2}
			-\frac{1}{\lambda(\sigma)^2} \right|\leq C_0 \int^{\tau}_{\sigma}
			\delta(t) dt.
			\endaligned
		\end{equation*}
	\end{lemma}

	
	Thus, we have
	
	\begin{lemma}[\cite{MiaoWX:dynamic gHartree}] Let $u$ be a
		radial solution of \eqref{har} satisfying  \eqref{sub thresh case},
		\eqref{Noscattering} and \eqref{para-uniform-SGD}, and $\lambda(t)$
		be as in Lemma \ref{compactness}. Then there exists $C$   such that
		for all $t\geq 0$,
		\begin{align}\label{ExpDecay:Inte}
			\int^{+\infty}_t  \delta(s)ds   \leq & \; C  e^{-ct}.
		\end{align}
	\end{lemma}

	\subsection{Convergence of the modulation parameters}
	Now by Lemma \ref{virialestimate-delta:positive}, Lemma \ref{control:parameter} and \eqref{ExpDecay:Inte}, we can get that there exists $ \lambda_0 \in (0, +\infty)$ such that \begin{equation}\label{lambdalimit}
		\lim_{t\rightarrow+\infty}\lambda(t) =\lambda_{0} .
	\end{equation} 
	The argument is standard, and we can refer to Section 6.4 in \cite{MiaoWX:dynamic gHartree} for details.
	
	By \eqref{ExpDecay:Inte}, there exists $t_n\rightarrow +\infty$ such
	that
	$$\delta(t_n) \rightarrow 0.$$
	Fix such $\{t_n\}_{n\in \N}$. The similar proof as that of
	\eqref{diff:Gradient}, we can show that
	\begin{equation}\label{deltalimit}
		\aligned \lim_{t\rightarrow +\infty} \delta(t)=0.
		\endaligned
	\end{equation}
	
	Now for large $t$, $\alpha(t)$ is well defined by Lemma
	\ref{estmodulpara}, \eqref{para-uniform-SGD}, \eqref{lambdalimit},
	and \eqref{deltalimit}. Furthermore, we have
	\begin{align*}
		\delta(t) \approx \big| \alpha(t) \big| = & \big| \alpha(t)
		-\alpha(+\infty) \big| \notag\\
		\leq & C \int^{+\infty}_t \big| \alpha'(s)\big| ds \leq C
		\int^{+\infty}_t  \mu^2(s) \delta(s) ds \leq C e^{-ct}.
		\label{expdecay:pointwise}
	\end{align*}
	
	Finally, by Lemma \ref{estmodulpara}, \eqref{para-uniform-SGD},
	and \eqref{lambdalimit},  there
	exists $\theta_{\infty}$ such that for large $t$, we have
	\begin{align*}
		\Big| \theta(t) - \theta_{\infty}\Big| \leq C \int^{\infty}_t \big|
		\theta'(s) \big| ds \leq C \int^{+\infty}_t  \mu^2(s) \delta(s) ds
		\leq C e^{-ct}.
	\end{align*}

	\subsection{Scattering for the negative times.}
	\label{S:sub:scattering} By Corollary \ref{global}, we know that $u$
	is globally well-defined. Assume that $u$ does not scatter for the
	negative time. Then by the analogue estimates as those from Subsection
	\ref{subs:compact} to Subsection \ref{subs:expconv}, we have
	\begin{enumerate}
		\item[\rm (a)] there exists $\lambda(t)$, defined for $t\in \R$, such
		that the set \begin{align*} K:=\Big\{\big(u(t)\big)_{\lambda(t)},\;
			t\in \R \Big\}
		\end{align*}
		is pre-compact in $\dot H^1_{\text{rad}}(\R^N)$.
		
		\item[\rm (b)] there exists an decreasing sequence $t'_n \rightarrow
		-\infty$ as $n\rightarrow +\infty$, such that
		\begin{align*}
			\lim_{n\rightarrow +\infty} \delta(t'_n)=0.
		\end{align*}
		
		\item[\rm (c)] there is a constant $C>0$ such that if $\; -\infty <
		\sigma < \tau < \infty$,
		\begin{equation*}
			\aligned \int^{\tau}_{\sigma} \delta (t) dt \leq C \left(
			\sup_{\sigma\leq t\leq \tau}\frac{1}{\lambda(t)^2} \right) \big(
			\delta(\sigma) + \delta(\tau) \big).
			\endaligned
		\end{equation*}
		
		\item[\rm (d)] there is a constant $C>0$
		such that for any $\sigma, \tau $ with
		$\sigma+\frac{1}{\lambda(\sigma)^2} \leq \tau$, we have
		\begin{equation*}
			\aligned \left|\frac{1}{\lambda(\tau)^2}
			-\frac{1}{\lambda(\sigma)^2} \right|\leq C \int^{\tau}_{\sigma}
			\delta(t) dt.
			\endaligned
		\end{equation*}
	\end{enumerate}
	From (b)-(d), we know that
	\begin{align*}
		\lim_{t\rightarrow \pm \infty} \delta(t)=0
	\end{align*}
	and $\frac{1}{\lambda(t)^2}$ is bounded for $t\in \R$. Hence we have
	\begin{align*}
		\int^{\tau}_{\sigma} \delta(s) \;ds\leq C
		\Big(\delta(\sigma)+\delta(\tau) \Big).
	\end{align*}
	Let $\sigma \rightarrow -\infty$ and $\tau\rightarrow +\infty$, we
	have
	$$\delta(t)=0,\;  \forall \; t\in \R.$$
	Thus $u=W$ up to the symmetry of \eqref{har}, which  contradicts with
	$\big\|u_0\big\|_{\dot H^1} < \big\|W\big\|_{\dot H^1}$. \qed


	\section{Uniqueness and the classification result}\label{S:uniqueness}
	
	In this section, we follows the arguments in \cite{DuyMerle:NLS:ThresholdSolution, DuyMerle:NLW:ThresholdSolution} and \cite{MiaoWX:dynamic gHartree}. 
	
	\subsection{Exponentially small solutions of the linearized equation.} 
	
	Consider the radial functions $v(t)$ and $g(t)$
	\begin{equation*}
		\aligned v\in C^0([t_0, +\infty), \dot H^1),\;\;  g\in C^0([t_0,
		+\infty), L^{\frac{2N}{N+2}})\;\;\text{and}\;\; \nabla g \in N(t_0,
		+\infty)
		\endaligned
	\end{equation*}
	such that
	\begin{align*}
		\partial_t v + \mathcal{L}v =g, & \;\; (x,t)\in \R^N  \times [t_0,
		+\infty), \\
		\label{baseestimate} \big\|v(t)\big\|_{\dot H^1} \leq C e^{-c_1t}, &
		\;\;\;\; \big\|g(t)\big\|_{L^{\frac{2N}{N+2}}(\R^N)}+  \big\| \nabla
		g(t)\big\|_{N(t, +\infty)} \leq  Ce^{-c_2 t},
	\end{align*}
	where $0<c_1<c_2$. For any $c >0$, we denote by $c^-$ a positive
	number arbitrary close to $c$ and such that $0<c^{-}<c$.
	
	By the Strichartz estimate and Lemma \ref{summation}, we have
	\begin{lemma}[\cite{DuyMerle:NLS:ThresholdSolution, DuyMerle:NLW:ThresholdSolution, MiaoWX:dynamic gHartree}]\label{vstrich}
		Under the above assumptions, then for any $(q,r)$ with $\frac2q=N
		(\frac12 -\frac1r ) $, $q\in [2, +\infty]$, we have
		\begin{align*}
			\big\|v\big\|_{l(t,+\infty)}+\big\|\nabla v\big\|_{L^q(t,+\infty;
				L^r)} \leq C e^{-c_1t}.
		\end{align*}
	\end{lemma}

	Now we show that $v$ must decay almost as fast as $g$, except in the
	direction $\YYY_+$ where the decay is of order $e^{-e_0t}$.

	\begin{proposition}[\cite{DuyMerle:NLS:ThresholdSolution, DuyMerle:NLW:ThresholdSolution, MiaoWX:dynamic gHartree}]\label{improved decay}
		The following self-improving estimates hold.
		\begin{enumerate}
			\item[\rm(a)] if $c_2 \leq e_0$ or $e_0 < c_1$, then
			\begin{equation*}
				\aligned \big\|v(t)\big\|_{\dot H^1}  \leq C e^{-c_2^{-}t},
				\endaligned
			\end{equation*}
			\item[\rm(b)] If $c_2 > e_0$, then there exists $a\in \R$
			such that
			\begin{equation*}
				\aligned \big\|v(t)-ae^{-e_0t}\mathcal{Y}_+\big\|_{\dot H^1}  \leq C
				e^{-c_2^{-}t}.
				\endaligned
			\end{equation*}
		\end{enumerate}
	\end{proposition}

	\subsection{Uniqueness.}

	\begin{proposition}\label{uniqueness}
		Let $u$ be a radial solution of \eqref{har}, defined on $[t_0,
		+\infty)$, such that $E(u)=E(W)$ and
		\begin{equation}\label{uconvGSexponential}
			\aligned \exists \;c, C>0,\; \big\|u-W \big\|_{\dot H^1} \leq C
			e^{-ct}, \;\forall\; t\geq t_0 .
			\endaligned
		\end{equation}
		Then there exists a unique $a \in \R$ such that $u=U^a$ is the
		solution of \eqref{har} defined in Proposition
		\ref{existence:thresholdsolution}.
	\end{proposition}
	For the sake of completeness, We give the details.
	\begin{proof} Let $u=W+v$ be a radial solution of \eqref{har} for $t\geq
		t_0$ satisfying \eqref{uconvGSexponential}. Then $v$ satisfies the
		linearized equation \eqref{linearequat}, that is,
		\begin{equation*}
			\aligned
			\partial_t v + \mathcal{L} v=R(v),
			\endaligned
		\end{equation*}
		
		\noindent{\bf Step 1: } We will show that there exists $a\in \R$
		such that
		\begin{align}\label{uconvappsol:e0decay}
			\forall\; \eta>0,\; \big\|v(t)-ae^{-e_0t}\YYY_+\big\|_{\dot H^1} +
			\big\|v(t)-ae^{-e_0t}\YYY_+\big\|_{l(t, +\infty)} \leq C_{\eta}
			e^{-2^-e_0t}.
		\end{align}
		Indeed, we only need to show
		\begin{align}\label{vdecay:e0}
			\big\|v(t)\big\|_{\dot H^1} \leq C e^{-e_0t}, &
			\;\; \big\|R(v(t))\big\|_{L^{\frac{2N}{N+2}}(\R^N)}+  \big\| \nabla
			R(v(t))\big\|_{N(t, +\infty)} \leq  Ce^{-2e_0 t},
		\end{align}
		which together with Proposition \ref{improved decay} gives
		\eqref{uconvappsol:e0decay}. By Lemma
		\ref{linearoperator:prelimestimate} and Lemma \ref{linearstrich} it
		suffices to show the first estimate.
		
		By \eqref{uconvGSexponential}, Lemma
		\ref{linearoperator:prelimestimate} and Lemma \ref{linearstrich}, we
		have
		\begin{align*}
			\big\|R(v(t))\big\|_{L^{\frac{2N}{N+2}}(\R^N)}+  \big\| \nabla
			R(v(t))\big\|_{N(t, +\infty)} \leq  Ce^{-2c t}.
		\end{align*}
		Then Proposition \ref{improved decay} gives that
		\begin{align*}
			\big\|v(t)\big\|_{\dot H^1} \leq C \Big(  e^{-e_0t} + e^{-\frac32ct}\Big).
		\end{align*}
		If $ \dfrac32\;c\geq e_0$, we obtain the first inequality in
		\eqref{vdecay:e0}. If not, an iteration argument gives the first
		inequality in \eqref{vdecay:e0}.
		
		\noindent{\bf Step 2: } We will use the induction argument to show
		that
		\begin{align}\label{uconvappsol:fastdecay}
			\forall\; m >0,\; \big\|u(t)-U^a(t)\big\|_{\dot H^1} +
			\big\|u(t)-U^a(t)\big\|_{l(t, +\infty)} \leq   e^{-mt}.
		\end{align}
		
		By Proposition \ref{existence:thresholdsolution} and Step 1,
		\eqref{uconvappsol:fastdecay} holds with $m=\frac32e_0$. Let us
		assume that \eqref{uconvappsol:fastdecay} holds with $m=m_1 > e_0$,
		we will show that it also holds for $m=m_1 + \frac12 e_0$, which
		yields \eqref{uconvappsol:fastdecay} by iteration.
		
		If writing $u^a(t)=U^{a}-W$, then we have
		\begin{align*}
			\partial_t \big( v-u^a \big) + \mathcal{L} \big( v-u^a \big) =
			R(v)- R(u^a).
		\end{align*}
		By assumption, we have \begin{align*} \big\|v(t)-u^a(t)\big\|_{\dot
				H^1} + \big\|v(t)-u^a(t)\big\|_{l(t, +\infty)} \leq C  e^{-m_1t}.
		\end{align*}
		By Lemma \ref{linearoperator:prelimestimate} and Lemma
		\ref{summation}, we have
		\begin{align*}
			\big\|R(v(t))-R(u^a(t))\big\|_{L^{\frac{2N}{N+2}}(\R^N)}+  \big\|
			\nabla \big( R(v(t))-R(u^a(t))\big)\big\|_{N(t, +\infty)} \leq
			Ce^{-(m_1 +e_0) t}.
		\end{align*}
		By Proposition \ref{improved decay} and Lemma \ref{vstrich}, we have
		\begin{align*}
			\forall\; m >0,\; \big\|v(t)-u^a(t)\big\|_{\dot H^1} +
			\big\|v(t)-u^a(t)\big\|_{l(t, +\infty)} \leq C
			e^{-(m_1+\frac{3}{4}e_0)t}\leq   e^{-(m_1+\frac{1}{2}e_0)t}.
		\end{align*}

		\noindent{\bf Step 3: End of the proof. } Using
		\eqref{uconvappsol:fastdecay} with $m=(k_0 +1)e_0$, (where $k_0$ is
		given by Proposition \ref{existence:thresholdsolution}), we get that
		for large $t>0$,
		\begin{align*}
			\big\|u(t)-U^a_{k_0}(t)\big\|_{l(t, +\infty)} \leq   e^{-(k_0+\frac12)e_0t}.
		\end{align*}
		By the uniqueness in Proposition \ref{existence:thresholdsolution},
		we get that $u=U^a$, which concludes the proof. \end{proof}
	

	\begin{corollary}\label{cor:uniqueness}
		For any $a\not =0$, there exists $T_a \in \R$  such that
		\begin{equation*}
			\left\{ \begin{array}{rl} U^a(t)=W^{+}(t-T_a), & \text{if}\;\; a>0,\\
				U^a(t)=W^{-}(t-T_a), & \text{if}\; \; a<0,
			\end{array} \right.
		\end{equation*}
	\end{corollary}
	\begin{proof}Let $a\not =0$ and choose $T_a$ such that
		$|a|e^{-e_0T_a}=1$. By \eqref{difference:Energy}, we have
		\begin{align}\label{ua:translation}
			\big\|U^a(t+T_a)-&W - \text{sign} (a) e^{-e_0t}\YYY_+\big\|_{\dot
				H^1} \nonumber \\
			= & \big\|U^a(t+T_a)-W-a
			e^{-e_0(t+T_a)}\YYY_+\big\|_{\dot H^1} \leq   e^{-\frac32 e_0
				(t+T_a)} \leq C e^{-\frac32 e_0 t}. 
		\end{align}
		In addition, $U^a(t+T_a)$ satisfies the assumption of Proposition
		\ref{uniqueness}, this implies that there exists $\widetilde{a}$
		such that \begin{align*} U^a(t+T_a) =U^{\widetilde{a}}(t).
		\end{align*}
		By \eqref{ua:translation} and \eqref{difference:Energy}, we know
		that  $\widetilde{a}=1$ if $a>0$ and $\widetilde{a}=-1$ if $a=-1$,
		this completes the proof.\end{proof}
	\begin{remark}
		This shows that  $U^a(t)=W^\pm(t)$ up to  time translation. 
	\end{remark}

	\subsection{Proof of Theorem \ref{classification}.} Let us first show
	(a). Let $u$ be a radial solution to \eqref{har} on $[-T_-, T_+]$
	such that, replacing if necessary $u(t)$ by $\overline{u}(-t)$,
	$$E(u_0)=E(W), \big\|\nabla u_0 \big\|_{L^2} < \big\|\nabla W
	\big\|_{L^2},\;\text{and}\; \big\|u\big\|_{Z(0, T_+)}=+\infty.$$
	Then by Proposition \ref{expdecay:subcase}, we know that $T_-=T_+=\infty$, scatters in negative time and there exist $\theta_0\in \R, \mu_0>0$ and $c,
	C>0$ such that
	\begin{align*} \big\|u(t)-W_{\theta_0, \mu_0}\big\|_{\dot H^1} \leq
		C e^{-ct}\; \text{for any}\; t\geq 0.
	\end{align*}
	This shows that $u_{-\theta_0, \frac{1}{\mu_0}}$ fulfills the
	assumptions of Proposition \ref{uniqueness}. By $\big\|u\big\|_{\dot
		H^1} < \big\| W\big\|_{\dot H^1}$ and Corollary
	\ref{cor:uniqueness}, we know that there exist  $a<0$ and $T_a$ such
	that
	\begin{align*}
		u_{-\theta_0, \frac{1}{\mu_0}}(t)=U^a(t) = W^{-}(t-T_a),
	\end{align*}
	which  shows (a).
	
	(b) is a consequence of the variational characterization of $W$.

	The proof of (c) is similar to that of (a). Let $u$ be a radial
	solution of \eqref{har} defined on $[0, +\infty)$  (Replacing if
	necessary $u(t)$ by $\overline{u}(-t)$) such that
	\begin{align*}
		E(u)=E(W), \; \big\| \nabla u_0\big\|_{L^2}> \big\|\nabla
		W\big\|_{L^2},\; \text{and}\; u_0 \in L^2.
	\end{align*}
	then by Proposition \ref{expdecay:supercase}, there exist
	$\theta_0\in \R, \mu_0>0$ and $c, C>0$ such that \begin{align*}
		\big\|u(t)-W_{\theta_0, \mu_0}\big\|_{\dot H^1} \leq C e^{-ct}.
	\end{align*}
	Similar to the proof of (a), we know that there exist  $a>0$ and
	$T_a$ such that
	\begin{align*}
		u_{-\theta_0, \frac{1}{\mu_0}}(t)=U^a(t) = W^{+}(t-T_a).
	\end{align*}
	This shows (c). The proof of Theorem \ref{classification} is
	completed. \qed


	\appendix
	\section{Nondegeneracy of  energy solution of  equation \eqref{elliptic equation}}\label{appen regularity}
	From Remark \ref{rem:nondeg} (b) and (c) , it suffices to show $L^{\infty}$-regularity of $\dot H^1(\R^N)$ solution of equation \eqref{elliptic equation} to prove Proposition \ref{prop:nondegeneracy}. Now we use  the Moser iteration method in \cite{DiMeVald:book} to show this result in this appendix.  
	
	Recalling equation \eqref{elliptic equation},  energy solution $u\in \dot{H}^{1}(\R^N)$ satisfies 
	\begin{equation}\label{elliptic equat}
		-\Delta u =(p-1)\left( I_{\lambda}*W^{p}\right)W^{p-2}u+p \left[ I_{\lambda}*(W^{p-1}u)\right]W^{p-1}.
	\end{equation}
	Let $\beta\ge 1,\;T>0$, we define the cutoff function
	\begin{equation}\label{deffunction}
		\varphi(t)=\begin{cases}
			\beta T^{\beta-1}(t-T)+T^{\beta},&if \;t\ge T;\\
			|t|^{\beta},&if \;-T\le t\le T; \\
			\beta T^{\beta-1}(T-t)+T^{\beta},&if \;t\le -T.
		\end{cases}
	\end{equation}
	Since $\varphi$ is convex and Lipschitz, we have
	\begin{equation}\label{derivation}
		(-\Delta)^s \varphi(u)\le \varphi'(u)(-\Delta )^s u,\;s\in[0,1]
	\end{equation}
	in the weak sense. 
	By the equation \eqref{elliptic equat}, we get
	\begin{align*}
		\int_{\R^N}(-\Delta u) \varphi(u)\varphi'(u)\;dx=&\, (p-1)\int_{\R^N}\left( I_{\lambda}*W^{p}\right)W^{p-2}u\varphi(u)\varphi'(u)\;dx\\
		&\, \,  + p\int_{\R^N} \left[ I_{\lambda}*(W^{p-1}u)\right]W^{p-1}\varphi(u)\varphi'(u)\;dx.
	\end{align*}
	Using \eqref{derivation} and integration by parts, we get
	\begin{equation*}
		\int_{\R^N}(-\Delta u) \varphi(u)\varphi'(u)\;dx=\int_{\R^N} |\nabla u|^2(\varphi' (u))^2\;dx+\int_{\R^N}|\nabla u|^2\varphi(u)\varphi''(u)\;dx.
	\end{equation*}
	By  \eqref{deffunction}, we know that the last integral is nonnegative, therefore, we have
	\begin{align*}
		\int_{\R^N} |\nabla u|^2(\varphi' (u))^2\;dx\le &\, (p-1)\int_{\R^N}\left( I_{\lambda}*W^{p} \right)W^{p-2}|u||\varphi(u)\varphi'(u)|\;dx\\
		&\,\,  +p\int_{\R^N} \left[ I_{\lambda}*(W^{p-1}|u|)\right]W^{p-1}\varphi(u)|\varphi'(u)|\;dx.
	\end{align*}
	Applying the weak-Young inequality and noting that $W=c\left(\frac{t}{t^2+|x|^2}\right)^{\frac{N-2}{2}}\in L^p(\R^N),\; p>(N-2)/N$, we have 
	$$I_{\lambda}*W^{p}\in L^{\infty}(\R^N).$$
	Using that $\varphi(u)\varphi'(u)\le \beta u^{2\beta -1}$ and the H\"{o}lder inequality, we have \begin{equation*}
		\int_{\R^N}\left( I_{\lambda}*W^{p} \right)W^{p-2}|u||\varphi(u)\varphi'(u)|\;dx\le C\beta\int_{\R^N}|u|^{2\beta}.
	\end{equation*}
	Similarly, we can get 
	\begin{equation*}
		\int_{\R^N} \left[ I_{\lambda}*(W^{p-1}|u|)\right]W^{p-1}\varphi(u)|\varphi'(u)|\;dx\le C\beta\int_{\R^N}|u|^{2\beta}.
	\end{equation*}
	Thus, by the Sobolev inequality, we have 
	\begin{equation*}
		\|\varphi(u)\|_{L^{2^*}}^2\le \|\varphi(u)\|_{\dot{H}^1}^2=\int_{\R^N} |\nabla u|^2(\varphi' (u))^2\;dx\le C\beta\int_{\R^N}|u|^{2\beta},
	\end{equation*}
	where $2^*=\frac{2N}{N-2}$.
	Let $T\rightarrow \infty$, we obtain
	\begin{equation*}
		\left(\int_{\R^N} |u|^{2^*\beta}\right)^{\frac{2}{2^*}}\le C\beta\int_{\R^N}|u|^{2\beta},
	\end{equation*}
	where C changing from line to line, but independent of $\beta$. Therefore,
	\begin{equation}\label{iteration}
		\left(\int_{\R^N} |u|^{2^*\beta}\right)^{\frac{1}{2^*\beta}}\le (C\beta)^{\frac{1}{2\beta}}\left(\int_{\R^N}|u|^{2\beta}\right)^{\frac{1}{2\beta}}.
	\end{equation}
	Because $u\in \dot{H}^{1}(\R^N)$, we know $u\in L^{\frac{2N}{N-2}}$ by the Sobolev inequality. Therefore, we can take $\beta_1=\frac{N}{N-2}\ge1$. From now on, we can follow exactly the iteration argument. That is, we define $\beta_{m+1},m\ge1$, so that $$2\beta_{m+1}=2^*\beta_m.$$
	Thus,
	$$\beta_{m+1}=\left(\frac{2^*}{2}\right)^m\beta_1,$$
	and replacing in equation \eqref{iteration}, we know that
	\begin{equation*}
		\left(\int_{\R^N}|u|^{2\beta_{m+1}}\right)^{\frac{1}{2\beta_{m+1}}}=\left(\int_{\R^N}|u|^{2^*\beta_m}\right)^{\frac{1}{2^*\beta_m}}\le \left(C\beta_m\int_{\R^N}|u|^{2\beta_m}\right)^{\frac{1}{2\beta_m}}.
	\end{equation*}
	Defining $C_{m+1}:=C\beta_m$ and
	$$A_{m}:=\big(\int_{\R^N}|u|^{2\beta_{m+1}}\big)^{\frac{1}{2\beta_{m+1}}},$$
	we conclude that there exists a constant $C_0>0$ independent of $m$, such that
	$$A_{m+1}\le \prod_{k=1}^{m}C_k^{\frac{1}{2\beta_{k}}}A_1\le C_0A_1.$$
	Thus,  we obtain
	$$\|u\|_{L^{\infty}(\R^N)}\le C_0A_1<+\infty.$$
	This together with the result in \cite{LLTX:Nondegeneracy} completes the proof of Proposition \ref{prop:nondegeneracy}.
	\qed
	
	\section{Coercivity of $\Phi$ on $H^{\bot} \cap \dot H^1_{rad}$}\label{appen CoerH}
	In this Appendix, we follow the argument in \cite{MiaoWX:dynamic gHartree} and prove Proposition \ref{coerH}. We divide the
	proof into two steps.
	
	\noindent{\bf Step 1: Nonnegative.} We claim that for any  function
	$h \in \dot H^1$, $\big(h, W\big)_{\dot H^1}=0$, there exists
	\begin{equation*}
		\aligned    \Phi(h) \geq  0.
		\endaligned
	\end{equation*}
	
	Indeed, let
	\begin{equation*}
		\aligned    I(u)= \frac{\big\|\nabla u \big\|^{2p}_2}{\big\|\nabla
			W\big\|^{2p}_2} - \frac{\displaystyle \iint_{\R^N\times\R^N}
			I_{\lambda}(x-y)|u(x)|^{p}|u(y)|^p) \; dxdy }{
			\displaystyle \iint_{\R^N\times\R^N}
			I_{\lambda}(x-y)W(x)^{p}W(y)^p \; dxdy }.
		\endaligned
	\end{equation*}
	By Proposition \ref{SharpConstant}, we have
	\begin{equation*}
		\aligned \forall \; u \in \dot H^1, \quad I(u) \geq 0.
		\endaligned
	\end{equation*}
	
	Take $\displaystyle h\in \dot H^1 $ with $(W, h)_{\dot H^1}=0, \;
	\alpha \in \R$, we consider the expansion of $I(W+\alpha h)$ in
	$\alpha$ of order $2$. Note that
	\begin{equation*}
		\aligned \big\|\nabla (W+ \alpha h) \big\|^{2p}_2 = \big\|\nabla W
		\big\|^{2p}_2  \left( 1 + p \alpha^2 \frac{\big\|h\big\|^2_{\dot
				H^1}}{\big\|W\big\|^2_{\dot H^1} } + O(\alpha^4)\right),
		\endaligned
	\end{equation*}
	and
	\begin{align*}
		& \iint_{\R^N\times\R^N} I_{\lambda}(x-y)\big|W(x)+ \alpha h(x) \big|^p\big|W(y)+ \alpha h(y) \big|^p \; dxdy\\
		= &\iint_{\R^N\times\R^N} I_{\lambda}(x-y)W(x)^{p}W(y)^p \; dxdy \\
		&   + p\alpha    \iint_{\R^N\times\R^N} I_{\lambda}(x-y)\left(W(x)^{p}W(y)^{p-1}h_1(y)+W(y)^{p}W(x)^{p-1}h_1(x) \right)\; dxdy
		\\  &  + \; \alpha^2  \iint_{\R^N\times\R^N} I_{\lambda}(x-y)\Big[\left(\frac{p(p-1)}{2}W(x)^{p-2}h_1(x)^{2}+\frac{p}{2}W(x)^{p-2}h_2(x)^{2}\right)W(y)^{p}\\
		&\qquad\qquad+\left(\frac{p(p-1)}{2}W(y)^{p-2}h_1(y)^{2}+\frac{p}{2}W(y)^{p-2}h_2(y)^{2}\right)W(x)^{p}\\
		&\qquad\qquad+p^2W(x)^{p-1}h_1(x)W(y)^{p-1}h_1(y)\Big] \; dxdy  \\
		& +O(\alpha^3) \\
		= &\iint_{\R^N\times\R^N} I_{\lambda}(x-y)W(x)^{p}W(y)^p \; dxdy  \\
		& \times \Big( 1 +  \frac{\alpha^2 }{\big\|W\big\|^2_{\dot	H^1}}
		\iint_{\R^N\times\R^N} I_{\lambda}(x-y)\big[p(p-1)W(x)^{p-2}h_1(x)^{2}W(y)^{p}\\ &+pW(x)^{p-2}h_2(x)^{2}W(y)^{p}+p^2W(x)^{p-1}h_1(x)W(y)^{p-1}h_1(y)\big] \; dxdy+ O(\alpha^3) \Big),
	\end{align*}
	where we use the facts that
	\begin{align*}
		\int_{\R^N} (I_{\lambda}*W^p) W(y)^{p-1}h_1(y)\; dy =& \int_{\R^N} - \Delta W(y)\cdot h_1(y)dy
		=0 ,\\
		\int_{\R^N} (I_{\lambda}*W^p) W(x)^{p-1}h_1(x)\; dx =& \int_{\R^N} - \Delta W(x)\cdot h_1(x)dy
		=0, \\
		\iint_{\R^N\times\R^N} I_{\lambda}(x-y) W(x)^{p-2}h_1(x)^{2} dx= & \iint_{\R^N\times\R^N} I_{\lambda}(x-y) W(y)^{p-2}h_1(y)^{2} dx,\\
		\iint_{\R^N\times\R^N} I_{\lambda}(x-y)W(x)^{p}W(y)^p \; dxdy = &  \int_{\R^N} \big|\nabla W \big|^2
		dx.
	\end{align*}
	Therefore we have
	\begin{align*}
		I(W+\alpha h)  = &  \frac{ p \alpha^2}{\big\|W\big\|^2_{\dot
				H^1}} \Big(  \big\|h\big\|^2_{\dot H^1}-
		\iint_{\R^N\times\R^N} I_{\lambda}(x-y)\big[(p-1)W(x)^{p-2}h_1(x)^{2}W(y)^{p}\\
		&\qquad\qquad+W(x)^{p-2}h_2(x)^{2}W(y)^{p}+W(x)^{p-1}h_1(x)W(y)^{p-1}h_1(y)\big] \ dxdy \Big) \\
		& + O(\alpha^3).
	\end{align*}
	
	Since $I(W+\alpha h) \geq 0$ for all $\alpha \in \R$,  we have
	\begin{align*}
		\aligned 2\Phi(h)&= \int_{\R^N} \big|\nabla h \big|^2 dx -
		\iint_{\R^N\times\R^N} I_{\lambda}(x-y)\big[(p-1)W(x)^{p-2}h_1(x)^{2}W(y)^{p}\\
		&\qquad+W(x)^{p-2}h_2(x)^{2}W(y)^{p}+W(x)^{p-1}h_1(x)W(y)^{p-1}h_1(y)\big] \ dxdy\\
		& \geq 0.
		\endaligned
	\end{align*}
	
	\noindent{\bf Step 2: Coercivity.} We show that there exists a
	constant $c_*>0$ such that for any radial function $h\in H^{\bot}$
	\begin{equation*}
		\aligned  \Phi (h) \geq c_* \big\|h\big\|^2_{\dot H^1}.
		\endaligned
	\end{equation*}
	
	Note that $\Phi(h) = \Phi_1 (h_1) + \Phi_2(h_2)$,  where
	\begin{align*}
		\Phi_1 (h_1)
		: = & \frac12 \int_{\R^N} (L_+h_1) h_1  \; dx\\
		=& \frac12 \int_{\R^N} \big| \nabla h_1 \big| ^2 \;dx- \frac{p-1}{2}
		\int_{\R^N} (I_{\lambda}*W^p)W(x)^{p-2}h_1(x)^{2}\;dx \\
		& \qquad \qquad \qquad \quad- \frac{p}{2}\int_{\R^N}
		(I_{\lambda}*W^{p-1}h_1)W(x)h_1(x)\;dx,
	\end{align*}
	and
	\begin{align*}
		\Phi_2(h_2) :=  & \frac12 \int_{\R^N} (L_-h_2) h_2 \; dx \\
		=& \frac12 \int_{\R^N}
		\big| \nabla h_2 \big| ^2 \; dx- \frac12 \ \int_{\R^N} (I_{\lambda}*W^p)W(x)^{p-2}h_2(x)^{2}\;dx .
	\end{align*}
	By Step 1, $L_+$ is nonnegative on $\{W\}^{\bot}$ (in the sense of
	$\dot H^1$) and $L_-$ is nonnegative. We will deduce the coercivity
	by the compactness argument in \cite{MiaoWX:dynamic gHartree, Wein:85:Modulationary stability}.
	
	We first show that there exists a constant $c$ such that for any
	radial real-valued $\dot H^1$-function $h_1 \in \{W,
	\widetilde{W}\}^{\bot}$ (in the sense of $\dot H^1$),
	\begin{equation*}
		\aligned   \Phi_1 (h_1) \geq c \big\| h_1 \big\|^2_{\dot H^1}.
		\endaligned
	\end{equation*}
	
	Assume that the above inequality does not hold. Then there exists a
	sequence of real-valued radial $\dot H^1$-functions $\{f_n\}_n$ such
	that
	\begin{equation} \label{H-noncoercivity}
		\aligned f_n \in H^{\bot}, \quad \lim_{n \rightarrow +\infty}
		\Phi_1(f_n) =0, \quad \big\|f_n\big\|_{\dot H^1} =1.
		\endaligned
	\end{equation}
	Extracting a subsequence from $\{f_n\}$, we may assume that
	\begin{equation*}
		\aligned f_n \rightharpoonup f_* \; \text{in} \; \dot H^1.
		\endaligned
	\end{equation*}
	The weak convergence of $f_n \in H^{\bot}$ to $f_*$ implies that
	$f_*\in H^{\bot}$. In addition, by compactness, we have
	\begin{align*}
		&\iint_{\R^N\times\R^N} I_{\lambda}(x-y)W(y)^pW(x)^{p-2}|f_n(x)|^{2}\;dxdy\\ 
		\longrightarrow & \iint_{\R^N\times\R^N} I_{\lambda}(x-y)W(y)^pW(x)^{p-2}|f_*(x)|^{2}\;dxdy, \\
		&\iint_{\R^N\times\R^N}I_{\lambda}(x-y)W(x)^{p-1}|f_n(x)|W(y)^{p-1}|f_n(y)| \;dxdy \\
		\longrightarrow &\iint_{\R^N\times\R^N}I_{\lambda}(x-y)W(x)^{p-1}|f_*(x)|W(y)^{p-1}|f_*(y)| \;dxdy.
	\end{align*}
	Thus by the Fatou lemma, \eqref{H-noncoercivity} and Step 1, we have
	\begin{equation*}
		\aligned 0 \leq \Phi_1(f_*)\leq  \liminf_{n\rightarrow +\infty}
		\Phi_1(f_n)=0.
		\endaligned
	\end{equation*}
	$f_*$ is the solution to the following minimizing problem
	\begin{align*}
		0= \min_{f\in \Omega\backslash \{0\}   }
		\frac{\displaystyle\int_{\R^N}  L_+ f \cdot f \; dx}{\big\|f\big\|^2_{\dot H^1}},\quad \Omega=\left\{f\in \dot
		H^1_{rad}, (f, W)_{\dot H^1}= (f, \widetilde{W})_{\dot
			H^1}=0\right\}.
	\end{align*}
	Thus for some Lagrange multipliers $\lambda_0, \lambda_1$, we have
	\begin{align*}
		L_+f_*=\lambda_0 \Delta W + \lambda_1 \Delta \widetilde{W}.
	\end{align*}
	Note that $(W, \widetilde{W})_{\dot H^1}=0$ and
	$L_+(\widetilde{W})=0$, we have
	\begin{align*}
		0=\int_{\R^N} f_* L_+(\widetilde{W})=\int_{\R^N} (L_+
		f_*)\widetilde{W} =\lambda_1 \big\|\widetilde{W}\big\|^2_{\dot H^1}
		\Longrightarrow \lambda_1=0.
	\end{align*}
	Thus
	\begin{align*}
		L_+f_*=\lambda_0 \Delta W =-\lambda_0 \big(  |x|^{-4}*W^2 \big) W
		= \frac{\lambda_0}{2} L_+W.
	\end{align*}
	By $\text{Null}(L_+) = \text{span}\{\widetilde{W}\}$ in Lemma
	\ref{keyassumption},  there exists $\mu_1$ such that
	\begin{align*}
		f_*= \frac{\lambda_0}{2} W + \mu_1 \widetilde{W}.
	\end{align*}
	Using $f_* \in H^{\bot}$, we get $\mu_1=0$ and $\lambda_0=0$, therefore, we have
	\begin{equation*}
		\aligned   f_*=0,\;\; \text{and}\;\; f_n \rightharpoonup 0 \;
		\text{in} \; \dot H^1 .
		\endaligned
	\end{equation*}
	Now, by  compactness, we have
	\begin{equation*}
		\iint_{\R^N\times\R^N} I_{\lambda}(x-y)W(y)^pW(x)^{p-2}|f_n(x)|^{2}\;dxdy
		\longrightarrow 0, 
	\end{equation*}
	\begin{equation*}
		\iint_{\R^N\times\R^N}I_{\lambda}(x-y)W(x)^{p-1}|f_n(x)|W(y)^{p-1}|f_n(y)| \;dxdy \longrightarrow 0.
	\end{equation*}
	By $\Phi_1(f_n)\rightarrow 0$ in \eqref{H-noncoercivity}, we get
	that
	\begin{equation*}
		\aligned \big\|\nabla f_n \big\|_2 \rightarrow 0,
		\endaligned
	\end{equation*}
	which contradicts $\big\|f_n\big\|_{\dot H^1} =1$ in
	\eqref{H-noncoercivity}.
	
	Using the same argument, we can show that there exists a constant
	$c$, such that for any real-valued radial $\dot H^1$-function $h_2
	\in \{W\}^{\bot}$, we have $$\Phi_2 (h_2) \geq c \big\| h_2
	\big\|^2_{\dot H^1}. $$ This completes the proof. \qed

	\section{Spectral properties of the linearized operator}
	\label{appen-spectralprop} In this Appendix, we follow the argument in \cite{DuyMerle:NLS:ThresholdSolution,  DuyMerle:NLW:ThresholdSolution, MiaoWX:dynamic gHartree} to give the proof
	of Proposition \ref{spectral}.
	
	\subsection{Existence, symmetry of the eigenfunctions} Note that
	$$\overline{\mathcal{L}v}=-\mathcal{L}\overline{v},$$ so that if
	$e_0>0$ is an eigenvalue of $\mathcal{L}$ with the eigenfunction
	$\mathcal{Y}_+$, $-e_0$ is also an eigenvalue with eigenfunction
	$\mathcal{Y}_-=\overline{\mathcal{Y}}_+$.
	
	Now we show the existence of $\mathcal{Y}_+$. Let $\mathcal{Y}_1=\Re
	\mathcal{Y}_+, \mathcal{Y}_2 = \Im \mathcal{Y}_+$, it suffices to
	show that
	\begin{equation}\label{eigenf-system}
		-L_- \mathcal{Y}_2 =\; e_0 \mathcal{Y}_1, \quad
		L_+ \mathcal{Y}_1 = \; e_0 \mathcal{Y}_2.
	\end{equation}
	
	From the proof of the coercivity property of $\Phi$ on $H^{\bot}\cap
	\dot H^1_{rad}$ in Proposition \ref{coerH}, we know that $L_-$ on
	$L^2$ with domain $H^2$ is self-adjoint and nonnegative, By Theorem
	3.35 in \cite[Page 281]{kato:book}, it has a unique square root $
	(L_-)^{1/2} $ with domain $H^1$.
	
	Assume that there exists a function $f\in H^4_{rad}$ such that
	\begin{equation*}
		\aligned \mathcal{P} f=-e^2_0 f_1, \quad \mathcal{P} := \big(
		L_-\big)^{1/2} \big(L_+\big) \big( L_-\big)^{1/2}.
		\endaligned
	\end{equation*}
	Then taking
	\begin{equation*}
		\aligned \mathcal{Y}_1 : = \big( L_-\big)^{1/2} f, \quad
		\mathcal{Y}_2 := \frac{1}{e_0} \big(L_+\big) \big( L_-\big)^{1/2} f
		\endaligned
	\end{equation*}
	would yield a solution of \eqref{eigenf-system}, which implies the
	existence of the radial $\mathcal{Y}_{\pm}$ by the rotation
	invariance of the operator $\mathcal{L}$.
	
	It suffices to show that the operator $\mathcal{P}$ on $L^2$ with
	domain $H^4_{rad}$ has a strictly negative eigenvalue. Since
	$\mathcal{P}$ is a relatively  compact, self-adjoint, perturbation
	of $\big(-\Delta\big)^2$, then by the Weyl theorem
	\cite{HisSig:96:book, kato:book}, we know that
	\begin{equation*}
		\aligned \sigma_{ess}(\mathcal{P})= [0, +\infty).
		\endaligned
	\end{equation*}
	Thus we only need to show that $\mathcal{P}$ has at least one
	negative eigenvalue $-e^2_0$.
	\begin{lemma}
		\begin{equation*}
			\aligned
			\sigma_{-}(\mathcal{P}):= \inf\{\big(\mathcal{P}f , f
			\big)_{L^2}, f\in H^4_{rad}, \big\|f\big\|_{L^2}=1\}<0.
			\endaligned
		\end{equation*}
	\end{lemma}
	\begin{proof} Note that
		\begin{equation*}
			\aligned \big(\mathcal{P}f, f \big)_{L^2} = \big( L_+F ,
			F\big)_{L^2}, \quad F: = \big(L_-\big)^{1/2}f.
			\endaligned
		\end{equation*}
		Thus it suffices to find $F$ such that there exists $g \in
		H^4_{rad}$, $ F=(\Delta + V) g$ and
		\begin{equation}\label{L+negative}
			\aligned \big(L_+F, F \big)_{L^2} < 0,
			\endaligned
		\end{equation}
		where $V:=(I_{\lambda}*W^p)W^{p-2}$.
		
		We divide into two cases. 
		
		First assume that N=3, 4, in these cases $W\notin L^2$, we will use the  localization method. Let $W_a:=\chi(x/a)W(x)$, where $\chi$ is a smooth, radial positive function such that $\chi(r)=1$ for $r\leq1$ and $\chi(r)=0$ for $r\geq2$. We first claim 
		\begin{equation}\label{Eanegative}
			\exists a>0,\; E_a:=\int L_+W_a W_a<0
		\end{equation}
		Recall that $-\Delta W=(I_{\lambda}*W^p)W^{p-1}$. Thus
		\begin{align*}
			L_+W_a=&-(p-2)(I_{\lambda}*W^p)\chi(x/a)W^{p-1}-p\big[I_{\lambda}*(\chi(x/a)W^p)\big]W^{p-1}\\
			&-\frac{2}{a}\nabla \chi(x/a) \nabla W-\frac{1}{a^2}\Delta \chi(x/a) W.
		\end{align*}
		Hence
		\begin{align*}
			\int L_+W_a W_a=&-(p-2)\int(I_{\lambda}*W^p)\chi^2(x/a)W^{p}-p\int \big[I_{\lambda}*(\chi(x/a)W^p)\big]W^{p}\\
			&-\frac{2}{a}\int \nabla \chi(x/a) \nabla W W-\frac{1}{a^2}\int \Delta \chi(x/a) W^2.
		\end{align*}
		Let $A:=\frac{2}{a}\int \nabla \chi(x/a) \nabla W W,\;B:=\frac{1}{a^2}\int \Delta \chi(x/a) W^2$. According to the explicit expression \eqref{w function} of W, $W\leq C|x|^{-(N-2)}$ and $\nabla W\leq C|x|^{-(N-1)}$ at infinity, which gives $|(A)|+|(B)|\leq \frac{C}{a}$ if $N=3$, $|(A)|+|(B)|\leq \frac{C}{a^2}$ if $N=4$. Hence \eqref{Eanegative}.
		
		Let us fix $a$ such that \eqref{Eanegative} holds. Recall that W is not in $L^2$. Thus $\Delta+V$ is a self-adjoint operator on $L^2$, with domain $H^2$, and without eigenfunction. In particular the orthogonal of its range $R(\Delta+V)$ is {$0$}, and thus $R(\Delta+V)$ is dense in $L^2$. Let $\epsilon >0$, and consider $G_{\epsilon}\in H^2$ such that
		$$\big\|(\Delta+V)G_{\epsilon}-(\Delta+V-1)W_a\big\|_{L^2}\leq\epsilon.$$
		Taking
		$$F_{\epsilon}:= \big(L_- +1 \big)^{-1} L_- G_{\epsilon},$$
		we obtain $\big\|(\Delta+V-1)(F_{\epsilon}-W_a)\big\|_{L^2}\leq\epsilon$ which implies $\big\|F_{\epsilon}-W_a\big\|_{H^2}\leq\epsilon\big\|(\Delta+V-1)^{-1}\big\|_{L^2\to L^2}$. Hence for some constant $C_0$, we have
		\begin{equation*}
			\aligned \left| \int_{\R^N}  L_+   F_{\epsilon} \cdot F_{\epsilon}
			-\int_{\R^N}  L_+ W_a \cdot W_a \right| \leq C_0\epsilon.
			\endaligned
		\end{equation*}
		As a consequence of \eqref{Eanegative}, we get \eqref{L+negative} for $F=F_{\epsilon},
		\epsilon=\frac{-E_a}{2C_0}$, which shows the claim in the  case $N=3,4$.
		
		Assume now that $N\geq5$, so that $W\in L^2$ and more generally in all space $H^s(\R^N).$ In this case 
		\begin{equation*}
			\aligned  \text{Ran}(L_-) ^{\bot} = \text{Null}(L_-)
			=\text{span}\{W\},
			\endaligned
		\end{equation*}
		Thus
		\begin{equation}\label{RLn}
			\aligned \text{Ran}(L_-) = \{f\in L^2, (f, W)_{L^2} =0\}.
			\endaligned
		\end{equation}
		
		Note that $L_+$ is a self-adjoint compact perturbation of $-\Delta$
		and
		\begin{equation*}
			\aligned \big(L_+W, W \big)_{L^2} = -(2p-2) \int |\nabla W|^2 \; dx < 0,
			\endaligned
		\end{equation*}
		which implies that $L_+$ has a negative eigenvalue. Let $Z$ be the
		eigenfunction for this eigenvalue (it is radial by the minimax
		principle). Note that $L_+ \widetilde{W} =0$, then for any real
		number $\alpha$, we have
		\begin{equation}\label{ZW}
			\aligned E_0: = \int_{\R^d} L_+(Z+\alpha \widetilde{W} ) \cdot (Z+
			\alpha \widetilde{W} ) = \int_{\R^d} L_+  Z \cdot Z < 0.
			\endaligned
		\end{equation}
		Note that
		$$\big(\widetilde{W} , W\big)_{L^2} \not = 0,$$
		we can choose the real number $\alpha_1$ such that
		\begin{equation*}
			\aligned \big( Z + \alpha_1 \widetilde{W} , W \big)_{L^2} =0,
			\endaligned
		\end{equation*}
		which means that
		\begin{equation*}
			\aligned \left( (L_- +1) (Z+\alpha_1 \widetilde{W} ) , W
			\right)_{L^2} =& \left(
			Z+\alpha_1 \widetilde{W }  ,(L_-+1)  W \right)_{L^2}\\
			=& \big( Z+\alpha_1 \widetilde{W }
			,  W \big)_{L^2}=0.
			\endaligned
		\end{equation*}
		
		By \eqref{RLn}, for any $\epsilon>0$, there exists a function
		$G_{\epsilon} \in H^2_{rad}$ such that
		\begin{equation*}
			\aligned \big\|  L_-  G_{\epsilon} -\big(L_-+1\big) (Z+\alpha_1
			\widetilde{W })\big\|_{L^2} < \epsilon.
			\endaligned
		\end{equation*}
		Taking
		$$F_{\epsilon}:= \big(L_- +1 \big)^{-1} L_-   G_{\epsilon},$$
		we obtain that
		
		\begin{equation*}
			\aligned \big\| \big(L_- + 1\big) \big(F_{\epsilon} - (Z + \alpha_1
			\widetilde{W })\big)\big\|_{L^2} \leq \epsilon,
			\endaligned
		\end{equation*}
		which implies that
		\begin{equation*}
			\aligned \big\| F_{\epsilon} - (Z+\alpha_1\widetilde{W}
			)\big\|_{H^2} \leq \epsilon \big\| \big(L_- +1 \big)^{-1}\big\|_{L^2
				\rightarrow L^2}.
			\endaligned
		\end{equation*}
		Hence for some constant $C_0$, we have
		\begin{equation*}
			\aligned \left| \int_{\R^d}  L_+   F_{\epsilon} \cdot F_{\epsilon}
			-\int_{\R^d}  L_+ (Z+\alpha_1\widetilde{W} ) \cdot
			(Z+\alpha_1\widetilde{W}) \right| \leq C_0\epsilon.
			\endaligned
		\end{equation*}
		By \eqref{ZW}, we have \eqref{L+negative} for $F=F_{\epsilon},
		\epsilon=-\frac{E_0}{2C_0}$. \end{proof}

	\subsection{Decay of the eigenfunctions at infinity } By the bootstrap argument,
	we know that $\YYY_{\pm} \in C^{\infty}\cap H^{\infty}$. In fact, we
	have $\YYY_{\pm} \in \mathcal{S}$. By \eqref{eigenf-system}, it
	suffices to show that $\YYY_1\in \mathcal{S}$. Note that $\YYY_1$
	satisfies that
	\begin{align*}
		\left(e^2_0 + \Delta^2 \right)\YYY_1 = &  -p\Delta \left(
		I_{\lambda}*(W^{p-1}\YYY_1)W^{p-1} \right)- p \left(
		I_{\lambda}*W^p \right)\left(I_{\lambda}*(W^{p-1}\YYY_1)\right)W^{2p-3} \\
		& - (p-1)\Delta \left( (I_{\lambda}*W^p)W^{p-2}\YYY_1 \right)  -
		(p-1)\left( (I_{\lambda}*W^p)^2 W^{2p-4}\YYY_1\right)\\
		& - (I_{\lambda}*W^p) W^{p-2} \Delta\YYY_1.
	\end{align*}
	Thus,
	\begin{align*} \left(e_0 - \Delta \right)^2 \YYY_1 =& - (p-1)\Delta \left( (I_{\lambda}*W^p)W^{p-2}\YYY_1 \right)-(p-1)\left( (I_{\lambda}*W^p)^2 W^{2p-4}\YYY_1\right)\\
		&-p\Delta \left( I_{\lambda}*(W^{p-1}\YYY_1)W^{p-1} \right)- p \left(
		I_{\lambda}*W^p \right)\left(I_{\lambda}*(W^{p-1}\YYY_1)\right)W^{2p-3}\\
		& - (I_{\lambda}*W^p) W^{p-2} \Delta\YYY_1-2e_0 \Delta\YYY_1.
	\end{align*}
	Because of the existence of the nonlocal interaction on the right
	hand side, the decay estimate in
	\cite{DuyMerle:NLS:ThresholdSolution} does not work. From the Bessel
	potential theory in \cite{Ste:70:book}, we know that the integral
	kernel $G$ of the operator $\left(e_0 - \Delta \right)^{-2}$ is
	\begin{align*}
		G(x)=\frac{1}{(4\pi)^2}\int^{\infty}_0e^{-\frac{e_0}{4\pi}\delta}
		e^{-\frac{\pi}{\delta}|x|^2} \delta^{\frac{-N+4}{2}}
		\frac{d\delta}{\delta}.
	\end{align*}
	Hence we have
	\begin{enumerate}
		\item $G(x)=\frac{|x|^{-N+4}}{\gamma(4)}+o(|x|^{-N+4}),\; |x|\rightarrow
		0$;
		\item there exists $c>0$ such that
		$$G(x)=o(e^{-c|x|}), \; |x|\rightarrow +\infty.$$
	\end{enumerate}
	Then the conclusion follows by the analogue estimates in
	\cite{FroJL:07:effective dynamics, KriMR:mass-subcritical har}.\qed
	
	\def\cprime{$'$}

\end{document}